\documentclass[12pt]{amsart}
\pdfoutput=1
\usepackage{amsmath, amsthm}
\usepackage{amssymb, bbm, dsfont}
\usepackage{mathbbol, bbold}

\usepackage[T1]{fontenc}

\usepackage{url, enumerate}
\usepackage{parskip}

\let\oldproof\proof
\let\endoldproof\endproof

\newenvironment{proofr}[1][]{
\renewcommand{\qed}{\nobreak \ifvmode \relax \else
      \ifdim\lastskip<1.5em \hskip-\lastskip
      \hskip1.5em plus0em minus0.5em \fi \nobreak
      \hfill $\blacktriangle_{#1}$
      \fi}%
\oldproof}{%
\endoldproof}

\let\proof\proofr
\let\endproof\endproofr

\def\endof{$\vphantom{i}$ \hfill $\blacktriangle$} 

\newtheorem{THEOREM}{Theorem}[section]

\newtheorem{Conclusion}[THEOREM]{Conclusion}

\newtheorem{Theorem}[THEOREM]{Theorem}
\newenvironment{theorem}{\begin{Theorem}}{\end{Theorem}}

\newtheorem{Lemma}[THEOREM]{Lemma}
\newenvironment{lemma}{\begin{Lemma}}{\end{Lemma}}

{\theoremstyle{definition}
\newtheorem{notation}[THEOREM]{Notation}
\newtheorem{definition}[THEOREM]{Definition}

\newtheorem{remark}[THEOREM]{Remark}

}

\newtheorem{Claim}[THEOREM]{Claim}
\newenvironment{claim}{\begin{Claim}}{\end{Claim}}

\newtheorem{Subclaim}[THEOREM]{Subclaim}
\newenvironment{subclaim}{\begin{Subclaim}}{\end{Subclaim}}

\newtheorem{Corollary}[THEOREM]{Corollary}
\newenvironment{corollary}{\begin{Corollary}}{\end{Corollary}}

\newtheorem{Corollary to the proof}[THEOREM]{Corollary to the proof}
\newenvironment{coroll-to-proof}{\begin{Corollary to the proof}}{\end{Corollary to the proof}}

\newtheorem{Proposition}[THEOREM]{Proposition}
\newenvironment{proposition}{\begin{Proposition}}{\end{Proposition}}

\newcommand{\sss}{\hskip2pt}
\newcommand{\ssss}{\hskip1pt}
\newcommand{\tupof}[2]{\langle\ssss #1\sss|\sss #2\ssss\rangle}
\newcommand{\tup}[1]{\langle\ssss #1\ssss\rangle}
\newcommand{\setof}[2]{\{\ssss #1\sss|\sss #2\ssss\}}
\newcommand{\set}[1]{\{\ssss #1\ssss\}}

\newcommand{\name}[1]{\dot{#1}}   
\newcommand{\on}{\upharpoonright}  
\newcommand{\forces}{\Vdash}

\newcommand{\thinks}{\models}

\newcommand{\pt}{\triangleleft}

\font\sevenrm = cmr7

\newcommand{\preq}{\prec}

\newcommand{\Pbb}{\mathbb{P}}         
\newcommand{\Qbb}{\mathbb{Q}}         

\newcommand{\onebb}{\mathbb{1}}

\newcommand{\cf}{\mathop{\rm cf}}   
\newcommand{\rge}{\mathop{\rm rge}} 
\newcommand{\dom}{\mathop{\rm dom}} 
\renewcommand{\lim}{\mathop{\rm lim}} 
\newcommand{\otp}{\mathop{\rm otp}} 
\newcommand{\id}{\mathop{\rm id}} 
\newcommand{\ssup}{\mathop{\rm ssup}} 
 
\renewcommand{\ni}{\notin}            
\newcommand{\concat}{\kern-.25pt\raise4pt\hbox{$\frown$}\kern-.25pt}
\newcommand{\Card}{\mathop{\rm Card}} 
\newcommand{\On}{\mathop{\rm On}} 
\newcommand{\lk}{\lower.04em\hbox{$<$}\kappa}
\newcommand{\lkplus}{\lower.04em\hbox{$<$}\kappa^+}
\newcommand{\lomegaone}{\lower.04em\hbox{$<$}\omega_1}

\newcommand{\cardrule}{\hrule height.2pt}
\newcommand{\cardrulefill}{\cleaders\cardrule\hfill}
\newcommand{\cardchar}[1]{\vbox{\ialign{##\crcr
    \cardrulefill\crcr\noalign{\kern1pt\nointerlineskip}
    $\hfil\displaystyle{#1}\hfil$\crcr}}}
\newcommand{\card}[1]{\cardchar{\cardchar{#1}}}

\newcommand{\obar}[1]{\cardchar{#1}}

\font\ninerm=cmr9  \font\eightrm=cmr8  \font\sixrm=cmr6
\font\ninebf=cmbx9 \font\eightbf=cmbx8 \font\sixbf=cmbx6
\font\ninett=cmtt9 \font\eighttt=cmti8
\font\nineit=cmmi9 \font\eightit=cmmi8 \font\sixit=cmmi6
\font\fiveit=cmmi5
\font\ninesl=cmsl9 \font\eightsl=cmsl8
\font\ninesy=cmsy9 \font\eightsy=cmsy8 \font\sixsy=cmsy6
\font\nineex=cmex9 \font\eightex=cmex8 

\catcode`@=11
\newskip\ttglue

\newcommand{\ninepoint}{\def\rm{\fam0\ninerm}
  \textfont0=\ninerm \scriptfont0=\sixrm \scriptscriptfont0=\fiverm
  \textfont1=\nineit  \scriptfont1=\sixit  \scriptscriptfont1=\fiveit
  \textfont2=\ninesy \scriptfont2=\sixsy \scriptscriptfont2=\fivesy
  \textfont3=\nineex  \scriptfont3=\sixex \scriptscriptfont3=\tenex
  \textfont\itfam=\nineit \def\it{\fam\itfam\nineit}%
  \textfont\slfam=\ninesl \def\sl{\fam\slfam\ninesl}%
  \textfont\ttfam=\ninett \def\tt{\fam\ttfam\ninett}%
  \textfont\bffam=\ninebf \scriptfont\bffam=\sixbf
   \scriptscriptfont\bffam=\fivebf \def\bf{\fam\bffam\ninebf}%
  \tt \ttglue=.5em plus.25em minus.15em
  \normalbaselineskip=11pt
  \setbox\strutbox=\hbox{\vrule height8pt depth3pt width0pt}%
  \let\sc=\sevenrm \let\big=\ninebig \normalbaselines\rm}

\newcommand{\nar}{\advance\leftskip by 40pt \advance\rightskip by 40pt}

\newcommand{\eightpoint}{\def\rm{\fam0\eightrm}
  \textfont0=\eightrm \scriptfont0=\sixrm \scriptscriptfont0=\fiverm
  \textfont1=\eightit  \scriptfont1=\sixit  \scriptscriptfont1=\fiveit
  \textfont2=\eightsy \scriptfont2=\sixsy \scriptscriptfont2=\fivesy
  \textfont3=\eightex  \scriptfont3=\tenex \scriptscriptfont3=\tenex
  \textfont\itfam=\eightit \def\it{\fam\itfam\eightit}%
  \textfont\slfam=\tensl \def\sl{\fam\slfam\eightsl}%
  \textfont\ttfam=\eighttt \def\tt{\fam\ttfam\eighttt}%
  \textfont\bffam=\eightbf \scriptfont\bffam=\sixbf
   \scriptscriptfont\bffam=\fivebf \def\bf{\fam\bffam\eightbf}%
  \tt \ttglue=.5em plus.25em minus.15em
  \normalbaselineskip=10pt
  \setbox\strutbox=\hbox{\vrule height8pt depth3pt width0pt}%
  \let\sc=\sevenrm \let\big=\eightbig \normalbaselines\rm}

\newcommand{\squaresub}[1]{\square_{\sss #1}{\hskip2pt}}

\title[Mitchell style forcing and morasses]{Mitchell-style forcing, with 
small working parts and collections of models as side conditions, and
gap-one simplified morasses}

\begin{document}

\author[Morgan]{Charles Morgan}

\thanks{The author thanks EPSRC for their support through grant
EP/I00498 and the Rheinische Friedrich-Wilhelms-Universität Bonn, 
UCL and UEA for institutional support.}

\address{Department of Mathematics,
University College London, Gower St.,
London, WC1E 6BT, UK}
\email{charles.morgan@ucl.ac.uk}

\begin{abstract}
We give a modification of Mitchell's technique (\cite{Mitchell1}-\cite{Mitchell4}) for
adding objects of size $\omega_2$ with conditions with finite working parts in which
the collections of models used as side conditions are very highly structured, arguably 
making them more wieldy. We use one such forcing (essentially a `pure side conditions' forcing)
to answer affirmatively the question, asked independently by Shelah and Velleman 
in the late 1980s, as to whether a $(\kappa^+,1)$-simplified morass
can be added by a forcing with working parts of size $<\kappa$.
\end{abstract}



\maketitle

\section*{Introduction}

\setcounter{section}{1}

The technique of \emph{forcing} is a way of moving from one collection/universe of sets/model of set theory
to another. Each forcing allows one to pass from any collection of sets possessing
the properties required to fuel the forcing to a new one. The destination model
will share many of the features of the point of departure, including satisfying the same basic
axioms (usually ZFC or ZF) and having the same ordinals, but will also have 
some differences. 

Forcings are typically designed with specific desired features of the 
destination models in mind. While the starting point either may not or assuredly will not have these features,
the destination model is guaranteed to have them. On the other hand, one usually wants
to control at the same time whether, or the extent to which, other 
important features of the original model are lost on the voyage;
\emph{being a cardinal} is a prominent example. 

The constituent elements of many forcings can be thought of informally as each
breaking down into a part which gives `positive' information about a
desired destination model feature and another which gives constraints,
limitations or `negative' information -- often aimed at ensuring that key origin model 
features are not lost. Such descriptions can be useful ways of thinking, albeit
that they should not be taken too literally.
These `positive' parts have also been called \emph{working
parts}.\footnote{The first use of the phrase `working parts' in this
context of which I am aware is in \cite{Todorcevic}.}  
Notwithstanding the fact that everything in an element of a forcing does some work,
this evocative designation has gained currency and we use it here.
\vskip6pt

A fascinating phenomena discovered relatively early in the development
of forcing was that some individual desired combinatorial features can be forced into
existence by forcings with differing sizes of working part. In
particular, some uncountable structures of size $\omega_1$ can be
forced not only by using conditions with countable working parts, but also, 
seemingly less likely intuitively, perhaps, by ones with finite working parts.

An early example is that of adding Souslin trees.  Tennenbaum's
original proof using a partial order consisting of finite conditions
can be contrasted with Jech's forcing with countable initial segments.
Each forcing has a property that ensures the $\omega_1$ of an origin model
remains a cardinal in the corresponding destination model:
Tennenbaum's has the countable chain condition, while Jech's is countably closed.

Another instructive example is that of adding closed unbounded subsets to stationary
subsets of $\omega_1$. Baumgartner, Harrington and Kleinberg 
used countable putative initial segments as conditions; the forcing is
countably closed. Baumgartner then gave a forcing with finite
conditions and which uses \emph{side conditions} to restrict the way in which
the working parts can be extended.  This partial order does not have
the countable chain condition, but cardinals are preserved as the
forcing is proper.

Somewhat later it was seen that even certain structures of size
$\omega_2$ can sometimes be added by forcing with finite working
parts. Ensuring the preservation of cardinals in such arguments,
however, is more delicate.  Often, as in \cite{Roitman},
\cite{Baumgartner-Shelah}, \cite{Velickovic}, \cite{M96} and later
papers, it is necessary to use auxiliary functions definable from
structured objects available in the ground model to do so.  In fact,
some of the most efficacious of these objects are not available in all
ground models and must either themselves be added by forcing or be
derivable from the internal structure of the ground model.

Eventually this approach was subsumed by a more powerful one,
introduced by Koszmider (\cite{Koszmider} -- see also \cite{Friedman},
\cite{M06}, \cite{M*clubs}). Koszmider's work adapted Todorcevic's
elaboration, in terms of $\in$-chains of models, of Baumgartner's idea 
of using side conditions to restrict the possible extensions of forcing conditions, 
to the structured objects mentioned above.

A key difference is that in Koszmider's technique the structured objects are sampled
dynamically during the forcings, in contrast to fixed, static uses in
the earlier papers. For some time it was reasonable to think that this was the
`right' way to do such forcing while avoiding collapsing cardinals by
using `collections of models as side conditions.'

Mitchell (\cite{Mitchell1}-\cite{Mitchell4}) then produced remarkable
arguments where the structure on the collection of models, whilst
still intricate, was produced as part of the forcing. One needs almost
nothing to power the forcing apart from some weak cardinal arithmetic
constraints (for example, $2^{\omega}<\lambda$ when one wants to
preserve $\omega_1$ and some larger $\lambda$).

These arguments should be seen as being more dynamic again.  Not only
are the collections of models used as side conditions chosen on the
fly, but the structured collection from which they are chosen is also
generated spontaneously. 

It would be good to understand better the properties of the
structured collections added as a by-product of the constructions'
main objectives when using Mitchell's technique, and to see to what
extent those properties rely on the collections being dynamically
assembled, and when and whether, in contrast, static axiomatizations
of them could be useful.

Mitchell places constraints on the collections of models
which form the side conditions which are visually strongly reminiscent
of the patterns formed by the ranges of maps in gap-one simplified morasses --
see, for example, \cite{Mitchell2}, Diagrams 1 and 2. This resemblence
speaks to the prescience of Velleman's definition of gap-one simplified
morasses (\cite{Velleman-SM}). Nevertheless, this paper does not address the 
relationship between gap-one simplified morasses and the structured collections of models 
added in the course of Mitchell's arguments directly. 
 \vskip6pt

Instead we show that one can use Mitchell-style forcing arguments to
add gap-one simplified morasses themselves.  This 
answers affirmatively the question, asked independently by Shelah and Velleman 
in the late 1980s, as to whether a $(\kappa^+,1)$-simplified morass
can be added by a forcing with working parts of size $\lk$.

However, rather than use
the constraints on collections of models used as side conditions that
Mitchell does, here the structure we impose on such collections is a
minaturization of that of a gap-one simplified morass.

Hence the paper also shows that one can successfully run
Mitchell-style arguments where the structured collections of models
added in the course of the construction not only in some ways
resembles a gap-one simplified morass, but actually gives rise to one.
It seems reasonable to hope that the extensive analysis of gap-one
simplified morasses already carried out may prove useful in pushing
through further Mitchell-style arguments in which the collections of
side conditions, as here, actually give rise to a gap-one simplified
morass.

In the forcings discussed so far in this introduction the models in the collections
constituting side conditions are all of the same size, for example they will all be countable 
if the working parts are all finite. In recent work, Neeman and others 
(\cite{Neeman}, \cite{GM}, \cite{Gitik} and so on) have used 
forcings with small working parts and collections of models of more than one size 
as side conditions to prove extentions of the proper forcing axiom and results in 
cardinal arithmetic.
In work in preparation we will build on the results of this paper and discuss 
relations between such forcings and higher gap simplified morasses.
\vskip12pt

This paper is self-contained, and no external material concerning
simplified morasses is needed. However, there is an
extensive literature on gap-one simplified morasses to which the reader
looking for background material can refer, including, {\it inter
alia\/}, \cite{Velleman-SM}, \cite{Velleman-SMLL}, \cite{Donder}, \cite{M96}, \cite{M06}. 
\vskip12pt

We now give a more concrete outline of the remainder of the paper.
We start by mentioning some of the notation that will be used, then
remind the reader of the definition of simplified morasses and conclude by 
sketching the forcing we shall use and summarizing its key properties.
\vskip6pt
Most of our set theoretic notation is standard.

\begin{notation}\label{standard-notation} 
$\On$ is the class of ordinals and $\Card$ is the collection of infinite cardinals. 
The letter $\kappa$ will always denote a  cardinal.
 If $X$ is a set and $\kappa$ a cardinal we
write $[X]^{\kappa}$ for $\setof{Y\subseteq X}{\card{Y}=\kappa}$ and
$[X]^{<\kappa}$ for $\setof{Y\subseteq X}{\card{Y}<\kappa}$.  
If $a$ is a set of ordinals then $\ssup(a)$ is the least ordinal $\xi$ such
that $a\subseteq\xi$.  
For $\alpha$, $\beta\in \On$ with $\alpha<\beta$ we write 
$f:\alpha\longrightarrow_{\hbox{\sevenrm o.p.}} \beta$ to indicate that 
$f$ is an order preserving function from $\alpha$ to $\beta$.
\end{notation}

We often use the following (standard) notation as short-hand.

\begin{definition}\label{notn-set-of-op-maps} For  $\alpha$, $\beta\in\On$
with $\alpha<\beta$ we write $(\beta)^{\alpha}$ for the set 
$\setof{f}{f:\alpha\longrightarrow_{\hbox{\sevenrm o.p.}}\beta}$ -- equivalently
the set of increasing sequences of length $\alpha$ from $\beta$,
and set $(\beta)^{<\alpha}=\bigcup_{\gamma<\alpha} (\beta)^{<\gamma}$.
\end{definition}

The following definition is also used frequently.

\begin{definition}\label{defn-amalgamation-pair} For $\theta$, $\phi\in \On$
with $\tau<\theta$ we say 
${\mathcal F}=\set{\id, h}$, $\subseteq (\theta)^{\tau}$, is an {\it amalgamation pair\/}
if there is some  $\sigma<\tau$
such that $h\on\sigma=\id$,
for all $\xi$ such that $\sigma+\xi<\tau$ we have
$h(\sigma+\xi)=\tau+\xi$, and $\tau\cup h``\tau$ is an initial segment of $\phi$.
In this case we say that $\sigma$ is the \emph{splitting point} of ${\mathcal F}$.
We say ${\mathcal F}$ is \emph{exact} if $\phi=\tau\cup h``\tau$ and \emph{almost exact} if 
$\phi= (\tau\cup h``\tau) + 1$.
\end{definition}
\vskip12pt

We recall the definition of an $(\kappa^+,1)$-simplified morass.
Parsing the definition below, a simplified morass consists of
an increasing sequence of length $\kappa^+$ of ordinals less than $\kappa^+$, with
$\kappa^{++}$ added at the top -- the \emph{levels} of the simplified morass
-- and collections of maps from lower levels to higher ones. 

The remaining axioms specify: the possible maps from one level to its
successor, that pairs of maps up to limit levels `split', {\it i.e.\/},
factor through a common second factor, that maps
from one level to another factor at every intermediate level and there
there are only at most $\kappa$ many maps from any level to any other below
the top level. Finally, every element of $\kappa^{++}$ is in the range of
some map up to the top level. 
\vskip12pt

\begin{definition}\label{defn-omega1-SM} (\cite{Velleman-SM}) 
$\tup{\tupof{\theta_\alpha}{\alpha\le \kappa^+}, 
\tupof{{\mathcal F}_{\alpha\beta}}{\alpha\le \beta\le \kappa^+}}$ is an 
$(\kappa^+,1)$-{\it simplified morass\/}
if 
\begin{itemize}
\item[$\bullet$] 
$\tupof{\theta_\alpha}{i< \kappa^+}\in (\kappa^+)^{\kappa^+}$
and $\theta_{\kappa^+}=\kappa^{++}$
\item[$\bullet$] 
for each $\alpha\le \beta\le \kappa^+$ 
one has that ${\mathcal
F}_{\alpha\beta}\subseteq 
\setof{f}{f:\theta_\alpha\longrightarrow_{\sevenrm o.p.\/} \theta_\beta}$
and 
\begin{itemize}
\item[$\bullet$] 
for each $\alpha\le \kappa^+$, \hskip10pt
${\mathcal F}_{\alpha\alpha} =\set{\id}$ 
\item[$\bullet$] 
for each $\alpha<\kappa^+$, \hskip10pt
\item[] \hfil ${\mathcal F}_{\alpha\alpha+1}$ is a singleton or an amalgamation pair \hfil
\item[$\bullet$] 
for each $\alpha\le \beta\le \gamma\le \kappa^+$, \hskip10pt
\item[] \hfil  ${\mathcal F}_{\alpha\gamma} = 
\setof{g\cdot f}{f\in {\mathcal F}_{\alpha\beta} \sss\&\sss 
g\in {\mathcal F}_{\beta\gamma}}$ \hfil
\item[$\bullet$] 
for each $\alpha\le \beta < \kappa^+$, \hskip10pt
$\card{{\mathcal F}_{\alpha\beta}}<\kappa^+$
\end{itemize}
\item[$\bullet$] 
if $\varepsilon\le \kappa^+$ is a limit ordinal then the simplified
morass is {\it directed at\/} $\varepsilon$ -- 
{\it i.e.\/}, if $\alpha$, $\beta<\varepsilon$, 
$e_\alpha\in {\mathcal F}_{\alpha\varepsilon}$ and
$e_\beta\in {\mathcal F}_{\beta\varepsilon}$ there is some $\gamma\in 
[\alpha\cup \beta, \varepsilon)$, 
$g\in {\mathcal F}_{\gamma\varepsilon}$, $f_i\in {\mathcal F}_{\alpha\gamma}$ and
$f_j\in {\mathcal F}_{\beta\gamma}$ such that $e_\alpha = g\cdot f_\alpha$ 
and $e_\beta=g\cdot f_\beta$
\item[$\bullet$] 
$\bigcup\setof{f``\theta_\alpha}{\alpha<\kappa^+ \sss\&\sss f\in {\mathcal
F}_{\alpha\kappa^+}} = \kappa^{++}$.
\end{itemize}
\end{definition}
\vskip6pt

Let $\lambda>\kappa^+$ be regular. A forcing will be given to add a
$(\kappa^+,1)$-simplified morass and make $\lambda=\kappa^{++}$ in the
generic extension. 

Conditions will consist of simplified morass segments of size less than $\kappa$, a map into
$\lambda$ (the map `up to the top'), which will also eventually be
one of the simplified morass maps, and a collection of elementary submodels of
$H_{\lambda}$, which combine well with the maps in the simplified morass
segment, as side conditions. The simplified morass segments in conditions 
in a generic for the forcing, having size smaller than $\kappa$, do not constitute initial segments 
of the added simplified morass: stronger conditions interleave more levels. 

In section \ref{proof-we-add-a-SM-section} we show that the forcing does indeed add an
$(\kappa^+,1)$-simplified morass.\footnote{The simplified morass added by the
forcing is not {\it neat\/} -- although, of course, as is well known,
Velleman (\cite{Velleman-SM}) showed how to obtain a neat gap-one simplified morass
(one in which $\bigcup\setof{f``\theta_0}{f\in {\mathcal F}_{0\kappa^+}}=\kappa^{++}$)
from an arbitrary one.} 
This allows us to observe that the forcing must also collapse all cardinals strictly between 
$\kappa^+$ and $\lambda$.

\vskip12pt

The forcing does not have the $\kappa^+$-chain condition, and indeed
this is inevitable, because a $(\kappa^+,1)$-simplified morass
provides one with a counterexample to the weak Chang conjecture and it
is well-known that no $\kappa^+$-cc forcing can add such a counterexample.

However the forcing does preserve many cardinals. We survey the
background material for these results in the next section, section
\ref{good-models-section}, and in section
\ref{preservation-of-kappaplus-section} we show that the forcing is
proper, so $\kappa^+$ is preserved.

In section \ref{preservation-of-lambda-and-cardinals-above-section} we
show that $\lambda$ and all cardinals greater than it are preserved.
The programmatic way to do this is to show that the forcing has the
variant of properness which preserves $\lambda$ provided
$2^{\kappa}<\lambda$ and has the $\lambda^+$-chain condition if
$2^{\kappa}\le\lambda$.  We give proofs in this style in order to make
clear how such proofs should run in forcing arguments where the
forcing conditions discussed here are used as side conditions.
However, as the forcing discussed in this paper is almost a pure side
conditions forcing, we are actually able to show that it has the
$\lambda$-chain condition provided $2^{\kappa}<\lambda$, and thus have
from this that $\lambda$ and all greater cardinals are preserved.

In section \ref{preservation-of-kappa-and-smaller-cardinals} we show that 
if $\kappa$ is a regular cardinal the forcing is $\lk$-closed and 
thus $\kappa$ and all smaller cardinals are preserved.
The proof is similar to the proof in 
section \ref{proof-we-add-a-SM-section}, easier in some places, essentially 
because we are dealing with a descending chain of conditions as opposed 
to a directed set, but a little more involved in others because of the 
technical constraints on limit levels of the working parts of the forcing conditions.
(Of course, section \ref{preservation-of-kappa-and-smaller-cardinals} can be skipped if one
is only interested in the case $\kappa=\omega$.)

Finally, in section \ref{results-proven}, we collect together the results proven.

\section{Background material on cardinal preservation and `good' models.}\label{good-models-section}
\vskip12pt

We start this section by reviewing some relevant definitions and results on cardinal preservation.

\begin{definition}\label{defn-predense} Let $(\Pbb,\le)$ be a partial order and $p\in \Pbb$. 
A set $D\subseteq \Pbb$ is \emph{predense below} $p$ if for all $q\le p$ there is some $r\in D$ such that
there is some $s\in \Pbb$ with $s\le q$, $r$. \end{definition}

\begin{definition}\label{defn-NP-generic} Let $\chi$ be a regular cardinal, $(\Pbb,\le)$ be a partial order and 
$(\Pbb,\le)\in  N\preq (H_\chi,\in)$. Then $p\in\Pbb$ is \emph{$(N,\Pbb)$-generic} if for every dense set
$D\subseteq \Pbb$ with $D\in N$ we have that $N\cap D$ is predense below $p$. \end{definition}

The following proposition is essentially from Hyttinen and Rautila's \cite{Hyttinen-Rautila};
Roslanowski and Shelah (\cite{Roslanowski-Shelah}) cite it as being due to
``folklore; Hyttinen-Rautila.'' We give the proof in detail for two reasons.
Firstly, the result in \cite{Hyttinen-Rautila} only covers preservation of successor 
cardinals and here we want a mild generalization 
covering both successor and limit cardinals. Secondly, 
\cite{Hyttinen-Rautila} is interested in iterable forcings whereas here
we only need worry about single step forcings and so can pare away some
restrictions on the forcings covered.

\begin{proposition}\label{properness-proof-that-lambda-is-preserved} 
(\emph{cf.}\cite{Hyttinen-Rautila}, 3.6) Let $\Pbb$ be a partial order, 
let $\lambda$ be an uncountable regular cardinal and let $(2^{\Pbb})^+ \le \chi$.
Suppose that $\Pbb$ is such that for all regular $\mu<\lambda$, there is a set 
${\mathcal U}\subseteq \setof{N \preq (H_\chi,\in) }{\mu\subseteq N \sss\&\sss \card{N} <\lambda}$ which is
unbounded in $[H_\chi]^{<\lambda}$ and such that for every $p\in\Pbb$ there is some 
$N \in {\mathcal U}$ with $p\in N$ and some $p^*\le p$ which is $(N,\Pbb)$-generic.
Then forcing with $\Pbb$ preserves that $\lambda$ is an uncountable regular cardinal.
\end{proposition}

\begin{proof}[\ref{properness-proof-that-lambda-is-preserved}] Let $\mu<\lambda$ be a cardinal.
Let $f:\mu\longrightarrow \On$ enumerate a set of ordinals in $V[G]$ of size $\mu$.  
Let $f=\name{f}^G$ for some $\Pbb$-name $\name{f}$. 
Let $\phi$ be the sentence 
$\forall \kappa<\lambda\sss \forall C\in [\ssup(\rge(f))]^\kappa\sss (\rge(f) \not\subseteq C)$. 
Suppose, towards a contradiction, $p\forces \phi$. 
Let $N\in {\mathcal U}$ with
$\name{f}$, $p$, $\Pbb\in N$, $\mu\subseteq N$ and $\card{N}=\kappa<\lambda$.

For each $\alpha<\mu$ let $D_\alpha=\setof{r\in \Pbb}{\exists \beta\sss r\forces \name{f}(\alpha)=\beta
\hbox{ or }r \hbox{ is incompatible} \discretionary{}{}{} \hbox{with }p}$.
Then for each  $\alpha<\mu$ we have that $D_\alpha$ is dense in $\Pbb$ and $D_{\alpha}$ is definable from
$\name{f}$, $p$ and $\Pbb$ and hence is an element of $N$. 

For each $\alpha<\kappa$ let 
$C_\alpha = \setof{\beta\in \On}{ \exists r\in\Pbb\cap N\sss r\forces  \name{f}(\alpha)=\beta}$,  
so $C_\alpha\subseteq N$.

By the hypothesis on $\Pbb$ let $p^*\le p$ be $(N,\Pbb)$-generic. Then, by the definition of 
$(N,\Pbb)$-genericity we have for each $\alpha<\mu$ that $N\cap D_\alpha$ is predense below 
$p^*$: if  $q\le p^*$ there is some $r\in N$ and some $\beta\in \On$ such that $r\forces \name{f}(\alpha)=\beta$
and there is some $s\in\Pbb$ with $s\le q$, $r$. Hence for all $\alpha<\mu$ and for all
$q\le p^*$ we have that there is some $s\le q$ such that $s\forces \name{f}(\alpha)\in C_\alpha$ 
(since $r$ forces this) and thus $p^*\forces \name{f}(\alpha)\in C_\alpha$.

Set $C=\bigcup_{\alpha<\mu}C_\alpha$. It is clear that $C\in V$ and 
from the previous paragraph we have that $p^*\forces \rge(\name{f})\subseteq C$.
Moreover, as $C_\alpha\subseteq N$ for each $\alpha<\mu$ we have that 
$\card{C}\le \mu.\kappa = \kappa <\lambda$. This gives the sought after contradiction.
\end{proof}

\begin{definition} If $\Pbb$ has the property of the statement of the definition for $\lambda=\omega_1$ it is
\emph{proper}. If  $\Pbb$ has the property of the statement of the definition for some successor of a regular 
cardinal $\lambda=\kappa^+$ it is (a mild generalization of being) \emph{$\kappa$-proper} in the sense of
\cite{Hyttinen-Rautila} 
\end{definition}

\vskip12pt

We shall also need some technical material on elementary submodels of $H_\lambda$ for regular 
$\lambda>\kappa^+$. 

\begin{lemma}\label{when-truncations-are-elementary}
Suppose $\lambda$, $\theta$ are regular cardinals with
$\kappa^+<\lambda<\theta$. Suppose $(N,\in)\preq (H_{\theta},\in)$ and 
either $\lambda\in N$ or
$\Card \cap N$ is unbounded in $\Card\cap \lambda$. 
Then $(N\cap H_\lambda)\preq H_\lambda$. \end{lemma}
\vskip6pt

\begin{definition}\label{defn-good-model} A structure $(N,\in)$ is a {\it good model\/} if 
$(N,\in)\preq (H_{\lambda},\in)$,
$\delta^N=\kappa^+\cap N<\kappa^+$, $\kappa^+\in N$ and 
$\card{N}=\kappa$. \end{definition}
 
\begin{lemma}\label{in-implies-subset-for-good-models} If $M$, $N$ are good models, $M\in N$ and
$\otp(M\cap\lambda)<\delta^N$ then $M\cap\lambda \subseteq N$. \end{lemma}

\begin{proof}[\ref{in-implies-subset-for-good-models}] As $M\in N$ we have $M\cap\On \in N$. Let
$\gamma=\otp(M\cap\lambda)$. As $H_\lambda\thinks \gamma=\otp(M\cap\lambda)$
and $N\preq H_\lambda$ one also has $N\thinks \gamma=\otp(M\cap\lambda)$. Since
$\gamma<\delta$ we have $\otp((M\cap\lambda)^N)=\otp(M\cap\lambda)$, and hence
$(M\cap\lambda)^N$, $=\setof{\alpha\in M\cap\lambda}{\alpha\in N}$, $=M\cap\lambda$.\end{proof}
\vskip6pt

Let $\lambda$ be a regular cardinal with $2^\kappa\le \lambda$.

By this assumption, let $\pt$ be a well-ordering of $(\lambda)^{<\kappa^+}$ in
order-type $\lambda$.
\vskip6pt

\begin{definition}\label{defn-very-good-model} $(N,\in,\pt)$ is a {\it very good model\/} if $(N,\in)$ is
a good model and $(N,\in,\pt)\preq (H_\lambda,\in,\pt)$.\end{definition}
\vskip6pt

\begin{lemma}\label{maps-in-element-small-model-are-in-larger-good-model}
Suppose $(M,\in,\pt)$ and $(N,\in\,\pt)$ are very good models,
$M\in N$ and $\otp(M\cap\lambda)<\delta^N$. Then
$((\lambda)^{<\kappa^+})^M\subseteq N$. \end{lemma}

(Observe that $((\lambda)^{<\kappa^+})^M
= \setof{f:\gamma\longrightarrow_{\hbox{\sevenrm o.p.}} \lambda}{\gamma<\delta^M\sss\&\sss f\in M}=$
$\setof{f:\gamma\longrightarrow_{\hbox{\sevenrm o.p.}} \lambda}{\gamma<\delta^M\sss\&\sss \rge(f)\subseteq 
M\cap\lambda}$.) 

\begin{proof}[\ref{maps-in-element-small-model-are-in-larger-good-model}] 
By Lemma (\ref{in-implies-subset-for-good-models}) we have $M\cap\lambda\subseteq N$. Let
$\tupof{x_\alpha}{\alpha<\lambda}$ enumerate $(\lambda)^{<\kappa^+}$. For all
$\alpha<\lambda$ we have that $\alpha\in M$ if and only if $x_\alpha\in
((\lambda)^{<\kappa^+})^M$, and similarly for $N$. Thus 
$((\lambda)^{<\kappa^+})^M\subseteq ((\lambda)^{<\kappa^+})^N$.\end{proof}
\vskip12pt

\begin{definition}\label{notn-trans-coll-of-a-good-model} If $N$ is a good model write 
$\obar{N}$ for the transitive collapse of $N$ and let $\pi^N:\obar{N}\longrightarrow N$ 
be the inverse of the transitive collapsing map.\end{definition}

\vskip6pt
\begin{lemma}\label{tr-coll-equiv-of-maps-in-good-models}  Suppose $N$ is a good model, 
$\gamma<\delta^N$, and $g\in(\lambda)^{\gamma}$ and $\rge(g)\subseteq N\cap\lambda$. 
Then $g\in N$ if and only if $(\pi^N)^{-1}\cdot g\in \obar{N}$.\end{lemma}

\begin{proof}[\ref{tr-coll-equiv-of-maps-in-good-models}] Immediate.\end{proof}

\vskip6pt
\begin{definition}\label{defn-fits} Let $\theta\in \On$ and let $f:\theta\longrightarrow\lambda$
be an order-preserving map. A good model $(N,\in)$ {\it fits\/} $f$
if and only if $f=\pi^N\on\On$. Equivalently,  $(N,\in)$ {\it fits\/}
$f$ if and only if $\rge(f)=N\cap\lambda$.\end{definition}

\vskip18pt

\section{Definition of the main forcing.}\label{defn-of-main-forcing-section}
\vskip12pt

\begin{definition}\label{defn-SMS}
$S=\tup{\tupof{\theta_i}{i\le \zeta}, \tupof{{\mathcal F}_{ij}}{i\le j\le \zeta}}$ is a 
{\it small\/} $(\kappa^+,1)$-{\it simplified morass segment\/} ({\it SMS\/})
if
\begin{itemize}
\item[$\bullet$] $\zeta<\kappa$ 
\item[$\bullet$] $\tupof{\theta_i}{i\le \zeta}$ 
$\in (\kappa^+)^{\zeta+1}$,  
\item[$\bullet$] if $i\le \zeta$ is a limit ordinal then $\cf(\theta_i)<\kappa$
\item[$\bullet$] for each $i\le j\le \zeta$ one has that ${\mathcal F}_{ij}\in [(\theta_j)^{\theta_i}]^{<\kappa}$,
and 
\begin{itemize}
\item[$\bullet$] for each $i\le \zeta$,
${\mathcal F}_{ii} =\set{\id}$ 
\item[$\bullet$] for each $i<\zeta$ either
\begin{itemize}
\item[$\bullet$] ${\mathcal F}_{ii+1}=\set{f_i}$ is a singleton and  $\ssup(\rge(f_i))<\theta_{i+1}$, or
\item[$\bullet$] ${\mathcal F}_{ii+1}=\set{\id, h_{i}}$ is an {\it almost exact amalgamation pair\/}
(as defined in Definition (\ref{defn-amalgamation-pair}))\end{itemize}
\item[$\bullet$] if $k\le \zeta$ is a limit ordinal then 
\begin{itemize}
\item[$\bullet$]  for all $i_0$, $i_1<k$ and 
$g_0\in {\mathcal F}_{i_0k}$, $g_1\in {\mathcal F}_{i_0k}$ there are $j\in [\max(\set{i_0,i_1}), k)$,
$h\in {\mathcal F}_{kj}$, $g'_0\in {\mathcal F}_{i_0j}$, $g'_1\in {\mathcal F}_{i_0j}$ such that
$g_0 = h \cdot g'_0$ and $g_1= h\cdot g'_1$.
\end{itemize}
\item[$\bullet$]
for each $i\le j\le k\le \zeta$ one has that
\begin{itemize}
\item[] ${\mathcal F}_{ik} = 
\setof{g\cdot f}{f\in {\mathcal F}_{ij} \sss\&\sss g\in {\mathcal F}_{jk}}$ 
\end{itemize}
\end{itemize}
\end{itemize}
\end{definition}\vskip6pt

For technical reasons (see the remarks in the paragraphs leading up to the definition of
$S$ in the proof of Proposition {\ref{less-than-kappa-closure}, ensuring that we can choose the 
$\theta^*_\xi$ suitably) we have defined small SMS in such a way as to 
make sure that the following lemma holds.

\begin{lemma}\label{individual-SMS-maps-are-not-cofinal} 
If $S$ is a small $(\kappa^+,1)$-SMS, $i < j\le \zeta$, and $f\in {\mathcal F}_{ij}$ then 
$f$ is not cofinal into $\theta_j$.\end{lemma}

\begin{proof}[\ref{individual-SMS-maps-are-not-cofinal}] The proof is by
induction on $j$ for each fixed $i$. If $j=i+1$ the result is immediate from 
the definition of ${\mathcal F}_{ii+1}$. (Recall that if ${\mathcal F}_{ii+1}=\set{\id,h_i}$ 
is an almost exact amalgamation pair then 
$\ssup(\rge(h_i))=\theta_i\cup h_i``\theta_i < \theta_{i+1}$.)
If $j=j'+1$ for some $j'>i$
then let $f=g\cdot h$ for some $h\in {\mathcal F}_{ij'}$ and $g\in {\mathcal F}_{j'j}$.
Again by the definition of ${\mathcal F}_{j'j}$ we have that $g$ is not cofinal and hence
neither is $f$. Finally, if $j$ is a limit ordinal then choose any $j'\in(i,j)$; then
$f=g\cdot h$ for some $h\in {\mathcal F}_{ij'}$ and $g\in {\mathcal F}_{j'j}$. By the inductive
hypothesis $h$ is not cofinal, hence neither is $f$. 
\end{proof}

\begin{lemma}\label{singletons-at-preds-of-origins-of-fitting-maps} 
If $S$ is a small $(\kappa^+,1)$-SMS, $0\le i < \zeta$, $f\in {\mathcal F}_{i+1\zeta}$, 
$F\in(\lambda)^{\theta_\zeta}$ and $N$ fits $F\cdot f$
then ${\mathcal F}_{ii+1}=\set{f_{i}}$.\end{lemma}

\begin{proof}[\ref{singletons-at-preds-of-origins-of-fitting-maps}] If not 
then $\otp(N\cap \lambda)=\sigma + \xi.2 + 1$ for some
$\xi$. But this is clearly impossible if $(N,\in)\preq
(H_{\lambda},\in)$.\end{proof}
\vskip6pt

Unsurprisingly, a variant of a generalization of the combination of
Velleman and Stanley's ubiquitiously helpful lemmas for simplified morasses
(\cite{Velleman-SM}, Lemma (3.2), and \cite{Velleman-SM}, Theorem (3.9), respectively) 
is very useful here.

\begin{lemma}\label{analogue-of-Vellemans-lemma}  Suppose 
$S=\tup{\tupof{\theta_i}{i\le \zeta}, \tupof{{\mathcal F}_{ij}}{i\le j\le \zeta}}$ 
is a small $(\kappa^+,1)$-SMS, $F\in(\lambda)^{\theta_n}$, $i\le j\le \zeta$, 
$f$, $g\in {\mathcal F}_{ij}$, and $\xi$, $\tau\le\theta_i$. 
If $\ssup(f``\xi)=\ssup(g``\tau)$ or if $j=n$ and 
$\ssup(F\cdot f``\xi)=\ssup(F\cdot g``\tau)$
then $\xi=\tau$ and $f\on \xi=g\on \xi$.\end{lemma}

\begin{proof}[\ref{analogue-of-Vellemans-lemma}] By induction on $j$ for each $i$. 
If $j=i+1$ the lemma is immediate from the two possible cases for the structure of 
${\mathcal F}_{ii+1}$. If $j=i'+1$ for some $i'>i$, note that one can factor $f$,
$g$ as $f=f_1\cdot f_0$ and $g=g_1\cdot g_0$ with $f_0$, $g_0\in {\mathcal F}_{ii'}$ 
and $f_1$, $g_1\in {\mathcal F}_{i'i'+1}$. If ${\mathcal F}_{i'i'+1}$
is a singleton then this instance of the lemma is immediate
from the lemma for $i$ and $i'$. Otherwise, ${\mathcal F}_{i'i'+1}$ is an
amalgamation pair with splitting point $\sigma_{i'}$ and either
$f_1=g_1$, in which case the lemma is again immediate by the induction
hypothesis, or $f_1\ne g_1$ and hence
$\ssup(f_0``\xi)=\ssup(g_0``\tau)\le\sigma_{i'}$ and so the lemma is
once again true. If $j\le \zeta$ is a limit ordinal then there is some
$k\in [i,j)$, $f' \in {\mathcal F}_{ik}$, $g' \in {\mathcal F}_{ik}$ and 
$h \in {\mathcal F}_{kj}$ such that $f=h\cdot f'$ and $g=h\cdot g'$. 
If $\ssup(f``\xi)=\ssup(g``\tau)$ then $\ssup(f'``\xi)=\ssup(g'``\tau)$.
By the inductive hypothesis the lemma is true for $k$, $f'$ and $g'$. 
Hence, applying $h$ on both sides, it is true of $j$, $f$ and $g$. 
If $j=\zeta$ then $\ssup(F\cdot f``\xi)=\ssup(F\cdot g``\tau)$ 
if and only if $\ssup(f``\xi)=\ssup(g``\tau)$, and so the lemma for 
$F\cdot f$ and $F\cdot g$ is true by the lemma for $f$ and $g$. 
$\hphantom{w}$ \end{proof}
\vskip18pt

We now introduce the main forcing notion that will be considered in this paper.
\vskip6pt

\begin{definition}\label{defn-main-forcing-notion} $\Pbb$ is the forcing notion in which
conditions other than $\onebb$ are of the form $p=(S^p,{\mathcal F}^p,{\mathcal N}^p)$ where 
\begin{itemize}
\item
$S^p=\tup{\tupof{\theta^p_i}{i\le \zeta^p}, \tupof{{\mathcal F}^p_{ij}}{i\le
j\le \zeta^p}}$ is a small $(\kappa^+,1)$-SMS, 
\item
${\mathcal F}^p=\setof{(F^p\cdot f, \theta^p_i)}{i\le \zeta^p\sss\&\sss f\in {\mathcal F}_{i\zeta^p}}$
for some $F^p\in(\lambda)^{\theta^p_{n^p}}$,
\item
${\mathcal N}^p$ is a set of very good models such that 
for every $M\in {\mathcal N}^p$ there is a unique 
$(i^M,f^M)\in \bigcup_{j\le \zeta^p}\set{j}\times {\mathcal F}^p_{j\zeta^p}$ 
such that $M$ fits $F^p\cdot f^M$ and
\begin{itemize}
\item[$\circ_1$] $\forall i\le j \le i^M\sss\sss ({\mathcal F}^p_{ij}\subseteq \obar{M})$ 
\item[$\circ_2$] $\forall N\in {\mathcal N}^p\sss\sss \forall\gamma<\delta^N\sss
\forall g\in (\theta_{i^N})^{\gamma}$
\begin{itemize}
\item[]
$(i^N\le i^M \sss\sss\&\sss\sss g\in \obar{N} \Longrightarrow
g\in \obar{M})$
\end{itemize}
\end{itemize}
\end{itemize}
\vskip6pt

We sometimes write instead, more informally, $p=(S^p,F^p,{\mathcal N}^p)$, 
as clearly ${\mathcal F}^p$ is recoverable from $F^p$ and $S^p$.

We take $\onebb$ to be the empty condition $(\emptyset,\emptyset,\emptyset)$ and, 
for numerological reasons, formally take $\zeta^{\onebb} = -1$.
\vskip6pt

The ordering is given by $q\le p$ if ($p=\onebb$ or)
\begin{itemize}
\item $\zeta^p\le \zeta^q$, 
\item there is some
$k:\zeta^p+1\longrightarrow_{\hbox{\sevenrm o.p.}} \zeta^q+1$ such that 
\begin{itemize}
\item[$\bullet$] for all $i\le \zeta^p$
one has $\theta_i^p=\theta^q_{k(i)}$, 
\item[$\bullet$] for all $i\le j\le \zeta^p$ one has
${\mathcal F}^p_{ij}\subseteq {\mathcal F}^q_{k(i)k(j)}$,
\item[$\bullet$] if $k(i+1)=k(i)+1$ then ${\mathcal F}^p_{ii+1}={\mathcal F}^q_{k(i)k(i)+1}$,
\end{itemize}
\item there is some $f^{pq}\in {\mathcal F}^q_{k(\zeta^p)\zeta^q}$ 
such that $F^p=F^q\cdot f^{pq}$,
\item ${\mathcal N}^p\subseteq {\mathcal N}^q$, 
\item if $N\in {\mathcal N}^q$ is witnessed by 
$(k(i),f^{pq}\cdot g)$ for some $g\in {\mathcal F}^p_{i\zeta^p}$,
then $N\in {\mathcal N}^p$ (and this is witnessed by $(i,g)$ since
$F^q\cdot f^{pq}\cdot g=F^p\cdot g$). 
\end{itemize}
\end{definition}

\vskip12pt

The clauses $\circ_1$ and $\circ_2$ are succinct and straightforward.
However it is useful, for the proofs below, to unpack them into a
sequence of conditions which involve the small SMS structure more
explicitly. 

\begin{lemma}\label{unpacking-the-circ-conditions} Suppose 
(as in Definition (\ref{defn-main-forcing-notion})) that 
$p=(S^p,{\mathcal F}^p,{\mathcal N}^p)$,\break  where 
$S^p=\tup{\tupof{\theta^p_i}{i\le \zeta^p}, \tupof{{\mathcal F}^p_{ij}}{i\le
j\le \zeta^p}}$ is a small $(\kappa^+,1)$-SMS, 
${\mathcal F}^p=
\setof{(F^p\cdot f, \theta^p_i)}{i\le \zeta^p\sss\&\sss f\in {\mathcal F}_{i\zeta^p}}$ for
some $F^p:\theta^p_{\zeta^p}\longrightarrow_{\sevenrm o.p.}\lambda$,
and ${\mathcal N}^p$ is a collection of very good models such that 
for all $M\in {\mathcal N}^p$ 
there is some (unique) $(i^M,f^M)\in \bigcup_{j\le \zeta^p}\set{j}\times {\mathcal F}^p_{j\zeta^p}$
such that $M$ fits $F^p\cdot f^M$. 

Then $p\in \Pbb$, 
{\it i.e.\/}, $\circ_1$ and $\circ_2$ hold, if and only if  the following conditions hold:
\begin{itemize}
\item[$\bullet_1$] $\forall i\le j < i^M \sss 
(\theta^p_i < \delta^M \sss\&\sss {\mathcal F}^p_{ij}\subseteq M )$
\item[$\bullet_2$] $\forall i< i^M\sss \forall g\in {\mathcal F}_{ii^M}
\sss (F^p\cdot f^M\cdot g\in M)$
\item[$\bullet_3$] $\forall N\in {\mathcal N}^p\sss \sss ( i^N\le i^M
\Longrightarrow ((\kappa^+)^{<\kappa^+})^N \subseteq ((\kappa^+)^{<\kappa^+})^M )$
\item[$\bullet_4$] $\forall N\in {\mathcal N}^p\sss \forall\gamma<\delta^N\sss
\forall g\in (\theta_{i^N})^{\gamma} (i^N=i^M \Longrightarrow$
\begin{itemize}
\item[] $(F^p\cdot f^N\cdot g \in N \Longleftrightarrow F^p\cdot f^M\cdot g \in M) )$
\end{itemize}
\item[$\bullet_5$] $\forall N\in {\mathcal N}^p\sss \forall\gamma<\delta^N\sss
\forall g\in (\theta_{i^N})^{\gamma} \sss\sss (i^N<i^M \Longrightarrow $
\begin{itemize}
\item[] $(F^p\cdot f^N\cdot g \in N \Longrightarrow g \in M) )$ 
\end{itemize}
\end{itemize}
\end{lemma}

\vskip12pt

\begin{proof}[\ref{unpacking-the-circ-conditions}]  For $M\in {\mathcal N}^p$ 
we have $\pi^M\on \On = F^p\cdot f^M$. 
Note that $\pi^M\on\delta^M=\id$, $\pi^M(\delta^M)=\kappa^+$ and that
$\forall i< i^M \sss\theta_i<\delta^M$. Observe that it is now immediate that 
for $i\in \set{1,2,3,4,5}$ we have that $\bullet_i$ holds 
if and only if $\bullet'_i$ does, where the $\bullet'_i$ are given by
\begin{itemize}
\item[$\bullet'_1$] $\forall i\le j <i^M\sss {\mathcal F}^p_{ij}\subseteq \obar{M}$ 
\item[$\bullet'_2$] $\forall i <i^M\sss {\mathcal F}^p_{ii^M}\subseteq \obar{M}$ 
\item[$\bullet'_3$] $\forall N\in {\mathcal N}^p\sss \forall\gamma<\delta^N\sss
\forall g\in (\delta^N)^{\gamma} \sss\sss (i^N\le i^M 
\Longrightarrow $
\begin{itemize}
\item[] $(g\in \obar{N} \Longleftrightarrow g\in \obar{M}))$
\end{itemize}
\item[$\bullet'_4$] $\forall N\in {\mathcal N}^p\sss \forall\gamma<\delta^N\sss
\forall g\in (\theta_{i^N})^{\gamma} \sss\sss (i^N=i^M \Longrightarrow $
\begin{itemize}
\item[] $(g\in \obar{N} \Longleftrightarrow g\in \obar{M}))$
\end{itemize}
\item[$\bullet'_5$] $\forall N\in {\mathcal N}^p\sss \forall\gamma<\delta^N\sss
\forall g\in (\theta_{i^N})^{\gamma} \sss\sss (i^N<i^M \Longrightarrow$ 
\begin{itemize}
\item[] $(g\in \obar{N} \Longrightarrow g\in \obar{M}))$
\end{itemize}
\end{itemize}

The conclusion now follows as $\circ_1$ is the conjunction of
$\bullet'_1$ and $\bullet'_2$ and $\circ_2$ is the conjuction on 
$\bullet'_3$, $\bullet'_4$ and $\bullet'_5$. \end{proof}

Before proceeding we give a useful lemma showing that for
a condition $p\in \Pbb$ the relationship of one member of the 
collection of good models ${\mathcal N}^p$ being an element of another
behaves well with respect to the SMS structure $S^p$ of $p$.\vskip6pt

\begin{lemma}\label{good-models-behave-well-with-SMS-structure} Suppose  
$p=(M,{\mathcal F},{\mathcal N})\in \Pbb$ and write $\zeta$ for $\zeta^p$.
If $N$, $N'\in {\mathcal N}$ with $i_{N'} <i_N$ 
and $N'\in N$ then $\exists f\in {\mathcal F}_{i_{N'}i_N} \sss 
(f_{N'} = f_N\cdot f )$.\end{lemma}

\begin{proof}[\ref{good-models-behave-well-with-SMS-structure}] By 
$\bullet_1$ for $p$ applied to $\theta_{i_N'}$ and $N$ one has
$\otp(N'\cap\lambda)<\delta^N$. By hypothesis $N'\in N$, so, by Lemma (1.7), 
$\rge(F\cdot f_{N'}) = N'\cap\lambda \subseteq N\cap\lambda$
$= \rge(F\cdot f_{N})$. As $F$ is order preserving
$\rge(f_{N'})\subseteq \rge(f_N)$. By the last property of ${\mathcal F}_{i_{N'}\zeta}$ given in
Definition (\ref{defn-main-forcing-notion}), we may factor 
$f_{N'}$ as $h\cdot f$ where $h\in {\mathcal F}_{i_N\zeta}$ and
$f\in {\mathcal F}_{i_{N'}i_N}$. Then $h\on \ssup(\rge(f)) =
f_N\on \ssup(\rge(f))$, by Lemma (\ref{analogue-of-Vellemans-lemma}), 
so $f_{N'} = f_N\cdot f$.\end{proof}
\vskip12pt

\vskip18pt

\section{Preservation of $\kappa^+$.}\label{preservation-of-kappaplus-section}\vskip12pt

\begin{proposition}\label{P-is-proper} $\Pbb$ is proper and hence forcing with
it preserves $\kappa^+$. \end{proposition}

\begin{proof}[\ref{P-is-proper}]  Let $x\in H_\chi$
and $p\in\Pbb$. Suppose $N$ is a very good model with
$x$, $p\in N$, $\setof{g\in (\theta^p_{i^M})^{\gamma}}{M\in {\mathcal N}^p 
\sss\&\sss \gamma<\delta^M \sss\&\sss F^p\cdot f^M\cdot g\in M } \subseteq N$, and  $\cf(N\cap \lambda)=\kappa$.
(It is easy to find such an $N$ by the downward Lowenheim-Skolem theorem.) 
 
As $\theta^p_{\zeta^p}\in N\cap\kappa^+$ one has that
$\theta^p_{\zeta^p}\subseteq N$. Since $F^p\in N$ we have that
$F^p``\theta^p_{\zeta^p}\subseteq N$. 

Let $\theta^*=\otp(N\cap\lambda)$, let $F^*$ be the inverse of the
transitive collapse of $N\cap\lambda$ and let $f^*={F^*}^{-1}\cdot F^p$. 
Now define $p^*=(S^*,{\mathcal F}^*,{\mathcal N}^*)$ by 
$S^*=\tup{ \tupof{ \theta^p_{i} }{ i\le \zeta^p }\concat \theta^*,
\tupof{ {\mathcal F}^p_{ij} }{ i\le j\le \zeta^p }\concat 
\tupof{ \setof{ f^*\cdot f }{ f\in {\mathcal F}^p_{i\zeta^p} } }{i\le \zeta^p} }$, 
${\mathcal F}^*=\setof{ (F^*\cdot f, \theta^{p^*}_i) }{ 
i\le \zeta^p+1 \sss\&\sss f\in {\mathcal F}^*_{i\zeta^p+1} }$ and 
${\mathcal N}^*={\mathcal N}\cup \set{N}$. 
\vskip6pt

\begin{lemma}\label{pstar-is-a-condition}  $p^*\in \Pbb$ \end{lemma}

\begin{proof}[\ref{pstar-is-a-condition}] The witnessing pair for $N$ is $(\zeta^{p^*},\id)$. 
The only thing to check is that $N$ satisfies $\bullet_1$ - $\bullet_5$.  

$\bullet_1$: It is immediate that ${\mathcal F}^*_{ij}\subseteq N$ for
all $i\le j\le \zeta^p$ as ${\mathcal F}^*_{ij}={\mathcal F}^p_{ij}$ for all $i\le j\le \zeta^p$
and ${\mathcal F}^p_{ij}\subseteq N$ since $p\in N$, $\zeta^p$ and each ${\mathcal
F}^p_{ij}$ has size $\lk$ and ${}^{<\kappa}N \subseteq N$.

$\bullet_2$: It is also immediate that $F^*\cdot f^*\cdot g$,
$=F^p\cdot g$, $\in N$ for all $g\in {\mathcal F}_{i\zeta^p}$ as for each $g\in {\mathcal
F}_{i\zeta^p}$ one has $g\in N$, as mentioned in the proof of $\bullet_1$, and $F^p\in N$ (since ${\mathcal
F}^p\in N$ has size $\lk$ and $^{<\kappa}N\subseteq N$, and $F^p=F^p\cdot \id\in {\mathcal F}^p$). 

$\bullet_3$ - $\bullet_5$: Lemma (\ref{maps-in-element-small-model-are-in-larger-good-model}) 
gives $\bullet_3$, $\bullet_4$ is vacuously
true, and $\bullet_5$ is true by the construction of $N$.
\end{proof} 
\vskip8pt

\begin{lemma}\label{pstar-is-NP-generic} $p^*$ is $(N,\Pbb)$-generic. \end{lemma}

\begin{proof}[\ref{pstar-is-NP-generic}] Let $q\le p^*$, and let $m^*\le \zeta^{q}$ be such that 
$\theta^q_{m^*} = \theta^*$. 

By the definition of small SMS we have two cases: $m^*=0$ or $m^*$ is a successor ordinal. 
For if $m^*$ were to be a limit ordinal we would have that $\cf(\theta^*)<\kappa$, contradicting 
the fact that $\cf(N\cap \lambda)=\kappa$. 

In the former case we take $q\on N= \onebb$.  Clearly in the case $q\on N\in \Pbb$. 

Now suppose (until the end of Claim (\ref{q-restricted-to-N-is-a-condition}))
that we are in the other case, so that $m^* =  m +1$ for some $m< \zeta^{q}$. 
In this case $\theta^q_{m+1}=\theta^*$.
Note, as remarked in Lemma (\ref{singletons-at-preds-of-origins-of-fitting-maps}), 
that we must have that ${\mathcal F}^q_{mm+1}$ is
the singleton $\set{f^q_{m}}$.

Define $q\on N$ as follows. Let $S^{q\on N}=
\tup{\tupof{\theta^q_i}{i\le m},\tupof{{\mathcal F}^q_{ij}}{i\le j\le m}}$.
Let $F^{q\on N}= F^*\cdot f^q_{m}$ and ${\mathcal F}^{q\on N}=
\setof{(F^*\cdot f^q_m \cdot f,\theta^q_i)}{i\le m \sss\&\sss 
f\in {\mathcal F}^q_{im}}$. Let ${\mathcal N}^{q\on N}={\mathcal N}^q \cap N$.  Set $q\on
N=(S^{q\on N},{\mathcal F}^{q\on N},{\mathcal N}^{q\on N})$.

\begin{claim}\label{q-restricted-to-N-is-a-condition} $q\on N\in\Pbb$. \end{claim}

\begin{proof}[\ref{q-restricted-to-N-is-a-condition}] We start by showing that the `${\mathcal F}$'-component 
of $q\on N$ is exactly the intersection of the `${\mathcal F}$'-component of $q$ with $N$.

\begin{subclaim}\label{maps-in-q-restr-N-are-exactly-the-maps-in-q-intersect-N} 
${\mathcal F}^{q\on N}= {\mathcal F}^q\cap N$.\end{subclaim}

\begin{proof}[\ref{maps-in-q-restr-N-are-exactly-the-maps-in-q-intersect-N} ] 
In one direction, if $i\le m$ and $f\in {\mathcal F}^q_{im}$ one has
that
$$F^{q\on N} \cdot f = F^* \cdot f^q_m \cdot f = F^q \cdot f^{p^*q} \cdot
f^q_m \cdot f,$$ where $(m+1,f^{p^*q})$ witnesses that $N\in {\mathcal N}^q$.
Consequently $F^{q\on N} \cdot f$ is of the form $F^q\cdot g$ for some $g\in
{\mathcal F}^q_{i\zeta^q}$ and $F^{q\on N} \cdot f\in N$ (by $\bullet_2$ for $N$ and
$q$). Thus ${\mathcal F}^{q\on N}\subseteq {\mathcal F}^q\cap N$.

In the other direction, suppose $i\le \zeta^q$, $g\in {\mathcal F}^q_{i\zeta^q}$ and
$(F^q\cdot g,\theta_i)\in N$. Then $i\le m$, since $\theta_i\in N$, and so $F^q\cdot
g``\theta_i\subseteq N$. Write $g$ as $g= l\cdot f^q_{m}\cdot f$ where $f\in
{\mathcal F}^q_{im}$ and $l\in {\mathcal F}^q_{m+1\zeta^q}$. Then $F^q\cdot l\on
(f^q_m \cdot f``\theta^q_i) = F^q\cdot f_N\on (f^q_m \cdot f``\theta^q_i)$,
where $(m+1,f_N)$ witnesses that $N\in {\mathcal N}^q$ (and hence $f_N$ is the map
$f^{p^*q}$), by Lemma (1.10).  Thus $F^q\cdot g = F^q\cdot f^{p^*q}\cdot
f^q_m \cdot f = F^{*}\cdot f^q_m\cdot f = F^{q\on N}\cdot f$.  Hence ${\mathcal
F}^q\cap N\subseteq {\mathcal F}^{q\on N}$.
\end{proof} 
\vskip6pt

Suppose $N'\in {\mathcal N}^q\cap N$.  Then, by Lemma (2.6), 
$f_{N'}=f_N\cdot f^q_m\cdot f$ for some
$f\in {\mathcal F}^q_{i_{N'}m}$.  Recalling that $f_{N}=f^{p^*q}$ and that $F^{q\on
N}=F^q\cdot f^{p^*q}\cdot f^q_m$, one thus has that $F^q\cdot f_{N'}= F^{q\on
N}\cdot f$. Hence $(i,f)$ is the (unique) witness that $N'\in {\mathcal N}^{q\on
N}$. Consequently $\bullet_1$ and $\bullet_2$ for $q\on N$ are immediate from these
properties for $q$. Clearly, $\bullet_3$ - $\bullet_5$ hold since they do for $q$.

Thus we have shown that $q\on N\in \Pbb$.\end{proof}
\vskip8pt 

Now let $s=(S^s,{\mathcal F}^s,{\mathcal N}^s)\in N\cap \Pbb$ be such that 
$s\le q\on N$. 

\begin{claim}\label{there-is-something-below-both-s-and-q} 
There is some $r\in\Pbb$ with $r\le q$, $s$. \end{claim}

\begin{proof}[\ref{there-is-something-below-both-s-and-q} ] Let $m=\zeta^{q\on N}$. Let
$k:m+1\longrightarrow \zeta^s+1$ witness that $s\le q\on N$. Extend $k$ so that
$\dom(k) =\zeta^q+1$ by setting $k(m+1+i)=\zeta^s+1+i$ for $i\in \otp((m, \zeta^q])$.

Set $\zeta^r=\zeta^s+ \otp((m, \zeta^q])$ and $\theta^r_i=\theta^s_i$ for $i\le \zeta^s$ and
$\theta^r_{\zeta^s+1+i}=\theta^q_{m+1+i}$ for $i\in \otp((m, \zeta^q])$.
The sequence $\tupof{\theta^r_i}{i<\zeta^r}$ is increasing because
$\theta_{\zeta^s}\in N$ and hence
$\theta_{\zeta^s}<\delta^N<\otp(N\cap\lambda)=\theta^q_{m+1}$.

Let $F^q\cdot f_N$ be the inverse of the transitive collapse of 
$\otp(N\cap \lambda)$, where $(m+1,f_N)\in {\mathcal F}^q_{m+1\zeta^q}$ witnesses
that $N\in {\mathcal N}^q$. 

Again, as $s\in N$ one has that $F^s\in N$ and hence
$F^s``\theta_{\zeta^s}\subseteq N$. Thus one can make the following
definition. 

For $i\le j\le \zeta^s$ let ${\mathcal F}^r_{ij}={\mathcal F}^s_{ij}$. Let 
${\mathcal F}^r_{\zeta^s\zeta^s+1} = \set{ f_N^{-1}\cdot ({F^q})^{-1}\cdot F^s}$.
Hence, in the notation for small SMSs from Definition (\ref{defn-SMS}), 
$f^r_{\zeta^s} = f_N^{-1}\cdot ({F^q})^{-1}\cdot F^s$.
For $1\le i\le j\le \otp((m, \zeta^q])$  let ${\mathcal F}^r_{\zeta^s+i\zeta^s+j}={\mathcal F}^q_{ij}$.
Finally, for $i\le j\le \zeta^r$ with $i\le \zeta^s$ and $j>\zeta^s$
set ${\mathcal F}^r_{ij} = \setof{ f \cdot f^r_{\zeta^s} \cdot g }{ 
g\in {\mathcal F}^r_{i\zeta^s} \sss\&\sss f\in {\mathcal F}^r_{\zeta^s+1j} }$.

Then $S^r= \tup{\tupof{\theta^r_i}{i\le \zeta^r}, 
\tupof{{\mathcal F}^r_{ij}}{i\le j\le \zeta^r}}$ is a small SMS. 

Next, set $F^r=F^q$ and ${\mathcal F}^r=
\setof{(F^r\cdot f, \theta_i)}{i\le \zeta^r\sss\&\sss f\in {\mathcal F}_{i\zeta^r}}$.

Lastly, set ${\mathcal N}^r={\mathcal N}^s\cup{\mathcal N}^q$. 

In order to finish the proof we must show that $\bullet_1$ - $\bullet_5$ hold
for $r$. If $M\in {\mathcal N}^s$ then $\bullet_1$ - $\bullet_5$ hold for $r$ since
they do for $s$. So suppose $M\in {\mathcal N}^q\setminus {\mathcal N}^s$. 

We shall write $i^M$ for $(i^M)^r$, $i^N$ for $(i^N)^r$, and so on,
in the remainder of the proof.

$\bullet_1$. Let $f\in {\mathcal F}^r_{ij}$ with $i\le j<i^M$. 

\emph{Case 1(i).} $i\le j< i^N \le i^M$. Then $f\in {\mathcal F}^s_{ij}$. As $s\in
N$ we have ${\mathcal F}^s_{ij}\in N$ and so $f\in N$ (because both $s$ and ${\mathcal
F}^s_{ij}$ are of size $\lk$). As $s\in N$ we have $\theta^s_j<\delta^N$ and so $f\in
((\kappa^+)^{<\kappa^+})^N$. As $i^N\le i^M$ we have 
$((\kappa^+)^{<\kappa^+})^N\subseteq ((\kappa^+)^{<\kappa^+})^M$ by $\bullet_3$
for $q$. Thus $f\in M$. \endof$_{\hbox{\sevenrm (Case 1(i))}}$

\emph{Case 1(ii).} $i\le j = i^N<i^M$. In this case we can factor $f$ as
$f=f_m\cdot g$ with $g\in {\mathcal F}^r_{im}={\mathcal F}^s_{im}$ and $f_m$ is the
(unique) map in ${\mathcal F}^r_{mm+1}$. By definition $f_m = (F^r\cdot
f^N)^{-1}\cdot F_s$, so $F^r\cdot f^N\cdot f_m=F^s$. Since $s\in N$ and $s$,
${\mathcal F}^s$ are of size $\lk$, we have that $F^s\in N$. As $\theta_{\zeta^s}<\delta^N$,
and by the case hypothesis that $i^N<i^M$, we
can apply $\bullet_5$ for $q$ (to $N$ and $M$), obtaining that $f_m\in M$. 
However, we also have that $g\in M$ by Case 1(i). Hence $f=f_m\cdot g\in M$. 
\endof$_{\hbox{\sevenrm (Case 1(ii))}}$

\emph{Case 1(iii).} $i<i^N<j<i^M$. In this case we can factor $f$ as $f=h\cdot
g$ with $g\in {\mathcal F}^r_{ii^N}$ and $h\in {\mathcal F}^r_{i^Nj}$. Then $h\in M$ by
$\bullet_1$ for $q$ and $g\in M$ by Case 1(ii). Hence $f\in M$.
\endof$_{\hbox{\sevenrm (Case 1(iii))}}$

\emph{Case 1(iv)} $i^N\le i\le j<i^M$. In this case 
${\mathcal F}^r_{ij}={\mathcal F}^q_{ij}$ and so $f\in M$ by $\bullet_1$ for $q$.
\endof$_{\hbox{\sevenrm (Case 1(iv)), (Verification of $\bullet_1$)}}$
\vskip6pt

$\bullet_2$. Let $i<i^M$ and $f\in {\mathcal F}^r_{ii^M}$. 

\emph{Case 2(i).} $i^N\le i$.  Then $\bullet_2$ for $f$ and $M$ holds by
$\bullet_2$ for $q$. \endof$_{\hbox{\sevenrm (Case 2(i))}}$

\emph{Case 2(ii).} $i<i^N<i^M$. Factor $f$ as $f=h\cdot f_m\cdot g$ where 
$h\in {\mathcal F}^r_{i^Ni^M}$ and $g\in {\mathcal F}^r_{im}$. Then $F^r\cdot f^M\cdot
h\in M$ by $\bullet_2$ for $q$. Also, $f_m\cdot g\in M$ by $\bullet_1$ for
$r$. So $F^r\cdot f^M\cdot h\cdot f_m\cdot g=F^r\cdot f^M\cdot f\in M$. 
\endof$_{\hbox{\sevenrm (Case 2(ii))}}$

\emph{Case 2(iii).} $i<i^N=i^M$. Factor $f$ as $f=f_m\cdot g$ where 
$g\in {\mathcal F}^r_{im}$. As in Case 2(ii), 
$g\in M$ by $\bullet_1$ for $r$. As in the proof of Case 1(ii), 
$F^r\cdot f^N\cdot f_m = F^s\in N$. So 
$F^r\cdot f^M\cdot f_m\in M$ by $\bullet_4$ for $q$. \hfill\break
\endof$_{\hbox{\sevenrm (Case 2(iii)), (Verification of $\bullet_2$)}}$
\vskip6pt

Now let $K\in {\mathcal N}^r$. If $K\in {\mathcal N}^q$ then each of $\bullet_1$ -
$\bullet_5$ holds for $r$ by the corresponding property for $q$. 
So suppose $K\in {\mathcal N}^s\setminus {\mathcal N}^q$. 

$\bullet_3$. By Lemma (\ref{maps-in-element-small-model-are-in-larger-good-model}) 
we have that $((\lambda)^{<\kappa^+})^K\subseteq N$. By $\bullet_3$ for $q$ we have 
$((\kappa^+)^{<\kappa^+})^N\subseteq ((\kappa^+)^{<\kappa^+})^M$. Thus
$((\kappa^+)^{<\kappa^+})^K\subseteq ((\kappa^+)^{<\kappa^+})^M$. 
\hfill\break \endof$_{\hbox{\sevenrm (Verification of $\bullet_3$)}}$

$\bullet_4$. As $i^K<i^M$ there is nothing to check here. 
\endof$_{\hbox{\sevenrm (Verification of $\bullet_4$)}}$

$\bullet_5$. Suppose $\gamma<\delta^K$, 
$g\in (\theta_{i^K})^\gamma$ and  $F^r\cdot f^K\cdot g\in K$. 
By Lemma (\ref{maps-in-element-small-model-are-in-larger-good-model}), 
since $s$, ${\mathcal N}^s$ and $K$ are all elements of $N$, we have 
$((\lambda)^{<\kappa^+})^K\subseteq N$, and so $F^r\cdot f^K\cdot g\in N$. 

We also have that $f^K=f^N\cdot f_m\cdot (f^K)^s$, so 
$F^r\cdot f^K = F^r\cdot f^N\cdot (f_m\cdot (f^K)^s)$. 
Thus $F^r\cdot f^K\in N$ by $\bullet_2$ for $r$ applied to $N$
and $f_m\cdot (f^K)^s$.

Hence $g$, $ = (F^r\cdot f^N)^{-1} \cdot (F^r\cdot f^N\cdot g)$,
$\in N$. But then $g\in ((\kappa^+)^{<\kappa^+})^N$ and so $g\in M$ by
$\bullet_3$ for $q$. \endof$_{\hbox{\sevenrm (Verification of $\bullet_5$)}}$
\vskip6pt

Thus $r=(S^r,{\mathcal F}^r,{\mathcal N}^r)\in \Pbb$.\end{proof} 
Hence we have shown that $p^*$ is $N$-generic.\end{proof} 
Thus forcing with $\Pbb$ preserves $\kappa^+$.\end{proof}

\vskip12pt

\section{Preservation of $\lambda$ and greater cardinals.}\label{preservation-of-lambda-and-cardinals-above-section}
\vskip12pt

We need with a couple of auxiliary definitions for this section.

\begin{definition}\label{defn-Zp-and-Xpi} For $p\in \Pbb$
let $Z^p=\setof{i\le \zeta^p}{\exists M\in {\mathcal N}^p \sss\sss (i=i^M)}$ and,
for $i\in Z^p$, $X_i^p= \setof{g\in (\theta_{i^M})^{\gamma}}{\exists M\in {\mathcal N}^p\sss (i=i^M\sss\&\sss
\gamma<\delta^M\sss\&\sss g\in \obar{M})} $.
\end{definition}

Note, for each $p\in\Pbb$ we have that $\forall i\le j\le
\zeta^p\sss(X^p_j\ne\emptyset \Longrightarrow X^p_i\subseteq X^p_j)$, by $\circ_2$,
and for $M\in {\mathcal N}^p$ that
$X^p_{i^M}=\setof{g\in (\theta_{i^M})^{\gamma}}{\gamma<\delta^M\sss\&\sss g\in \obar{M}}$.

For $p\in\Pbb$ and $i\in Z^p$ one has that $X_i^p$ is a
subset of $(\kappa^+)^{<\kappa^+}$ of size at most $\kappa$.

\vskip6pt
We start by showing that $\Pbb$ satisfies the criterion of Proposition 
(\ref{properness-proof-that-lambda-is-preserved}) for preserving $\lambda$. 
Referring to that criterion, here for each $p\in\Pbb$ we will be able to take
the relevant $(N,\Pbb)$-generic $p^*$ to be $p$ itself. 
Because of this, as we will see below, a similar argument allows us to show 
the stronger conclusion that $\Pbb$ has the $\lambda$-chain condition.
Nevertheless, we give the more general argument here as a template, because variants of it 
will be necessary when employing $\Pbb$ as a collection of side conditions in 
forcings $\Qbb$ (when the $(N,\Qbb)$-generic $p^*$ will in general 
not always be able to be taken to be the $p\in\Qbb$ itself) which cannot have the 
$\lambda$-chain condition. 

\begin{proposition}\label{forcing-with-P-preserves-lambda} Suppose $2^\kappa<\lambda$. 
Then forcing with $\Pbb$ preserves $\lambda$. \end{proposition}

\begin{proof}[\ref{forcing-with-P-preserves-lambda}] Let $x\in H_\chi$ and $\mu<\lambda$. 
Suppose that $D$ is a dense set in $\Pbb$ and $x$, $p$, $D$, $\Pbb\in N$
where $N\preq H_{\chi}$, $\mu\le \delta^N=N\cap\lambda\in\lambda$ and
$\cf(\delta^N)\ge\kappa^+$, $\card{N}<\lambda$, and $N^{\kappa}\subseteq N$. 

\begin{claim}\label{conds-in-D-below-p-compat-with-some-cond-below-p-in-D-cap-N} 
If $q\le p$ and $q\in D$ there is some
$s\in D\cap N$ such that $s\le p$ and $s$ and $q$ are compatible.\end{claim}

\begin{proof}[\ref{conds-in-D-below-p-compat-with-some-cond-below-p-in-D-cap-N}] 
As $\theta^q_{\zeta^q}<\kappa^+$ we have that  $F^q``\theta^q_{\zeta^q}$ is of size at most $\kappa$, 
so since $N^\kappa\subseteq N$ we have that $Y=N\cap F^q``\theta^q_{\zeta^q} \in N$. 
Also, as $N\cap \lambda\in \lambda$ we have that
$Y$ is an initial segment of $F^q``\theta^q_{\zeta^q}$. Also, again as $N^\kappa\subseteq N$, 
each $X_i^q$ for $i\in Z^q$ is an element of $N$. 

Let $\phi(x)$ be the formula which is the conjunction of the following
subformulae:
$x\in \Pbb$; $S^x=S^q$; $x\le p$; $Y$ is an initial segment
of $F^x``\theta^x_{\zeta^x}$; $Z^x=Z^q$; 
and $\tupof{X_i^x}{i\in Z^x}=\tupof{X_i^q}{i\in Z^q}$.

Since $q$ witnesses
$H_{\chi}\thinks \exists x \phi(x)$ there is some $s\in N$ such
that $N\thinks \phi(s)$. 

It is easy now to amalgamate the morass parts of $s$ and $q$ - our
comments about $Y$ have ensured that we have a head-tail-tail
amalgamation which we can use in the morass. In more detail, define $r$
by setting $\zeta^r=\zeta^q+1$, $S^r\on \zeta^r = S^q = S^s$, 
$\theta^r_{\zeta^r}=\otp(\rge(F^s)\cup\rge(F^q))+1$,
${\mathcal F}^r_{\zeta^q\zeta^r}=\set{\id, h^r_{\zeta^q}}$, where the splitting point of 
${\mathcal F}^r_{\zeta^q\zeta^r}$, $\sigma^r_{\zeta^q}$, is $\otp(Y)$, 
${\mathcal F}^r_{i\zeta^r}=\setof{f\cdot g}{f\in 
{\mathcal F}^r_{\zeta^q\zeta^r}\sss\&\sss g\in {\mathcal F}^r_{i\zeta^q}}$ for $i<\zeta^q$,
$F^r\on \otp(\rge(F^s)\cup\rge(F^q))$ is the inverse of the transitive collape of
$\rge(F^s)\cup\rge(F^q)$, $F^r (\otp(\rge(F^s)\cup\rge(F^q)))=\ssup(\rge(F^s)\cup\rge(F^q))$, 
and ${\mathcal N}^r={\mathcal N}^s\cup {\mathcal N}^q$.

\begin{subclaim}\label{r-is-in-P}
${\mathcal N}^r$ is such that $r\in \Pbb$, whence
$r$ witnesses that $s$ and $q$ are compatible. \end{subclaim}

\begin{proof}[\ref{r-is-in-P}] The property $\circ_1$ holds for $r$ 
immediately from $\circ_1$ for $q$ and $s$. We will be finished if we
show that $\circ_2$ also holds for $r$.

Suppose that $M\in {\mathcal N}^q$ and $K\in {\mathcal N}^s$. (In all other cases 
$\circ_2$ holds by virtue of holding for $q$ or $s$.)

Since $Z^s=Z^q$ there are $M'\in {\mathcal N}^s$ and $K'\in {\mathcal
N}^q$ such that $i^M=i^{M'}$ and $i^K=i^{K'}$ and, as $s$, $q\in \Pbb$
we have that $\circ_2$ holds for each of the two 
pairs $K$, $M'$ and $K'$, $M$. Thus, if $i^K\le i^M$ we
have that $X^q_{i^K}=X^s_{i^K}$ implies that $\circ_2$ holds for $K$,
$M$, while if $i^K\ge i^M$ then $X^q_{i^M}=X^s_{i^M}$ implies that $\circ_2$ holds
for $K$, $M$. \end{proof} Thus we have shown Claim 
(\ref{conds-in-D-below-p-compat-with-some-cond-below-p-in-D-cap-N}) \end{proof}
Hence forcing with $\Pbb$  preserves $\lambda$. \end{proof}

\begin{proposition}\label{preservation-cards-above-lambda} If $2^\kappa\le\lambda$ then 
cardinals above $\lambda$ are preserved. \end{proposition}

\begin{proof}[\ref{preservation-cards-above-lambda}] As $2^{\kappa}\le \lambda$ 
the forcing has the $\lambda^+$ chain condition.\end{proof}

Of course, if $\lambda$ is strongly inaccessible then clearly the forcing $\Pbb$ has the
$\lambda$-chain condition. However, as mentioned above in the preamble to 
Proposition (\ref{forcing-with-P-preserves-lambda}), we can in fact show that $\Pbb$ 
always has the $\lambda$-chain condition. 

\begin{proposition}\label{P-has-lambda-cc} If $2^\kappa<\lambda$ then 
$\Pbb$ has the $\lambda$-chain condition.
\end{proposition}

\begin{proof}[\ref{P-has-lambda-cc}] Suppose that ${\mathcal A}$ is an antichain in $\Pbb$ and 
$\card{\mathcal A}=\lambda$.
Let $\chi$ be large enough so that ${\mathcal A}$, $\Pbb\in H_\chi$ (\emph{e.g.}, $\chi = (2^{\lambda})^+$).
Let $N\preq H_{\chi}$ be such that $\delta^N=N\cap\lambda\in\lambda$ and
$\cf(\delta^N)\ge\kappa^+$, $\card{N}<\lambda$, and $N^{\kappa}\subseteq N$
and ${\mathcal A}$, $\Pbb\in N$. Choose $q\in {\mathcal A}\setminus N$.

\begin{claim}\label{conds-in-A-minus-N-with-some-cond-in-A-cap-N} 
There is some $s\in {\mathcal A}\cap N$ such that $s$ and $q$ are compatible.\end{claim}

\begin{proof}[\ref{conds-in-A-minus-N-with-some-cond-in-A-cap-N}]
As $\theta^q_{\zeta^q}<\kappa^+$ we have that  $F^q``\theta^q_{\zeta^q}$ is of size at most $\kappa$, 
so since $N^\kappa\subseteq N$ we have that $Y=N\cap F^q``\theta^q_{\zeta^q} \in N$. 
Also, as $N\cap \lambda\in \lambda$ we have that
$Y$ is an initial segment of $F^q``\theta^q_{\zeta^q}$.
As $N^\kappa\subseteq N$ we also have that each $X_i^q\in N$.

Let $\phi(x)$ be the formula which is the conjunction of the following
subformulae:
$x\in \Pbb$; $x\in {\mathcal A}$; $S^x=S^q$;  $Y$ is an initial segment
of $F^x``\theta^x_{\zeta^x}$; $Z^x=Z^q$; 
and $\tupof{X_i^x}{i\in Z^x}=\tupof{X_i^q}{i\in Z^q}$.

As $q$ witnesses $H_{\chi}\thinks \exists x \phi(x)$ there is some $s\in N$ such
that $N\thinks \phi(s)$. 

Define $r$ exactly as in the proof of Claim 
(\ref{conds-in-D-below-p-compat-with-some-cond-below-p-in-D-cap-N}):
by setting $\zeta^r=\zeta^q+1$, $S^r\on \zeta^r = S^q = S^s$, 
$\theta^r_{\zeta^r}=\otp(\rge(F^s)\cup\rge(F^q))+1$,
${\mathcal F}^r_{\zeta^q\zeta^r}=\set{\id, h^r_{\zeta^q}}$, where the splitting point of 
${\mathcal F}^r_{\zeta^q\zeta^r}$, $\sigma^r_{\zeta^q}$, is $\otp(Y)$, 
${\mathcal F}^r_{i\zeta^r}=\setof{f\cdot g}{f\in 
{\mathcal F}^r_{\zeta^q\zeta^r}\sss\&\sss g\in {\mathcal F}^r_{i\zeta^q}}$ for $i<\zeta^q$,
$F^r\on \otp(\rge(F^s)\cup\rge(F^q))$ is the inverse of the transitive collape of
$\rge(F^s)\cup\rge(F^q)$, $F^r (\otp(\rge(F^s)\cup\rge(F^q)))=\ssup(\rge(F^s)\cup\rge(F^q))$, 
and ${\mathcal N}^r={\mathcal N}^s\cup {\mathcal N}^q$.

\begin{subclaim}\label{r-is-in-P-cc-argt}
${\mathcal N}^r$ is such that $r\in \Pbb$, whence
$r\in \Pbb$ and $r\le q$, $s$. \end{subclaim}

\begin{proof}[\ref{r-is-in-P-cc-argt}] Verbatim as in the proof of Subclaim (\ref{r-is-in-P}).
\end{proof} Thus we have shown there is some
$s\in {\mathcal A}\cap N$ compatible with $q$.\end{proof}
This contradicts the assumption that ${\mathcal A}$ is an antichain. Thus we have shown that
$\Pbb$ contains no antichains of size $\lambda$.\end{proof}

\vskip18pt

\section{Preservation of $\kappa$ and all smaller cardinals}\label{preservation-of-kappa-and-smaller-cardinals}

In this section we show that if $\kappa$ is a regular cardinal then
$\Pbb$ is $\lk$-closed. The proof is similar to the proof in the following section that
forcing with $\Pbb$ adds a $(\kappa^+,1)$-simplified morass.

Suppose that $\kappa$ is a regular cardinal. 

\begin{proposition}\label{less-than-kappa-closure} 
The forcing with $\Pbb$ is $\lk$-closed.
\end{proposition}

\begin{proof}[\ref{less-than-kappa-closure}] Let $\mu$ be a cardinal with $\mu<\kappa$. 
Let $\tupof{p_\alpha}{\alpha<\mu}$ be a descending sequence 
of conditions from $\Pbb$. For $\alpha<\beta<\mu$ let $k_{\alpha\beta}$ witness that 
$p_\beta\le p_\alpha$. Note that 
if $\alpha<\beta<\gamma<\mu$ then $k_{\alpha\gamma}=k_{\beta\gamma}\cdot k_{\alpha\beta}$. 

We start by constructing a small SMS segment $S$ into which each $S^{p_\alpha}$ is embedded.

Define two relations $\sim$ and $<$ on the set 
$\setof{(\alpha,i)}{\alpha<\mu\sss\&\sss i\le \zeta^{p_\alpha}}$,
by $(\alpha,i)\sim (\beta,j)$ if $\alpha\le\beta$ and $k_{\alpha\beta}(i)=j$ or 
$\beta\le\alpha$ and $k_{\beta\alpha}(j)=i$, and, similarly, $(\alpha,i) < (\beta,j)$ 
and $\alpha\le\beta$ and $k(i)<j$ or $\beta\le\alpha$ and $k(j)<i$. Clearly $\sim$ is an
equivalence relation and $<$ is a linear ordering which respects
$\sim$, inducing a linear ordering on the $\sim$-equivalence classes
${\mathcal E}=\setof{[(\alpha,i)]_\sim}{\alpha<\mu\sss\&\sss i\le \zeta^{p_\alpha}}$.  

For $x=[(\alpha,i)]_\sim\in {\mathcal E}$ let $\theta_x = \theta^{p_\alpha}_i$. Note, by
the definition of $\sim$ and of $\le_{\Pbb}$, that $\theta_x$ is
independent of the choice of the pair $(\alpha,i)$. Note, also, that 
$x = [(\alpha,i)]_{\sim} < y = [(\beta,j)]_{\sim}$ if and only if
$\theta_x<\theta_y<\kappa^+$. Consequently, $<$ well-orders ${\mathcal E}$.

For $x$, $y\in {\mathcal E}$ define
${\mathcal F}_{xy}=\bigcup\setof{f\in {\mathcal F}^{p_\alpha}_{ij}}{\alpha<\mu\sss\&\sss 
x=[(\alpha,i)]_\sim \sss\&\sss y=[(\alpha,j)]_\sim }$.
Let ${\mathcal F}_x= \setof{F^{p_\alpha}\cdot f}{x=[(\alpha,i)]_\sim \sss\&\sss f\in {\mathcal F}^{p_\alpha}_{i\zeta^p_{\alpha}} }$.

Let $\setof{x_{\xi}}{\xi<\zeta'}$ be the $<$-increasing ordering of ${\mathcal E}$. 
Note that as $\card{\setof{(\alpha,i)}{\alpha<\mu\sss\&\sss i\le \zeta^{p_\alpha}}} < \kappa$ 
we must have that $\zeta'<\kappa$. 

If $\zeta'$ is a successor, \emph{i.e}, there is a $<$-maximal element of ${\mathcal E}$,
then set $\zeta$ to be the predecessor of $\zeta$. Otherwise, let $\zeta=\zeta'$. 

For $\alpha<\mu$ define $k_\alpha:\zeta^{p_{\alpha}}+1\longrightarrow \zeta+1$
by $k_\alpha(i)= $ that $\xi$ such that $x_\xi = [(\alpha,i)]_{\sim}$ for $i<\zeta^{p_{\alpha}}$.

For $\xi<\zeta'$ let $\theta_\xi=\theta_{x_\xi}$. For $\nu\le\xi<\zeta'$ let 
${\mathcal F}_{\nu\xi}={\mathcal F}_{x_\nu x_\xi}$. If $\zeta'=\zeta+1$ and $x_\zeta=[(\alpha,i)]$ for 
$\alpha<\mu$ and $i=\zeta^{p_\alpha}$, let  $F= F^{p_{\alpha}}$.

If $\zeta=\zeta'$ let $\theta_\zeta = \otp(\bigcup\setof{\rge(F^{p_\alpha})}{\alpha<\mu})$, let
$F$ be the inverse of the transitive collapse of $\bigcup\setof{\rge(F^{p_\alpha})}{\alpha<\mu}$
and for $\xi<\zeta$ let ${\mathcal F}_{\xi\zeta} = 
\bigcup\setof{F^{-1}\cdot F^{p_{\alpha}}\cdot f}{\alpha<\mu \sss\&\sss 
k_\alpha(i)=\xi \sss\&\sss f\in {\mathcal F}^{p_\alpha}_{i\zeta^{p_\alpha}} }$.

We shall show that the structure $\tup{\tupof{\theta_\xi}{\xi\le \zeta}, 
\tupof{{\mathcal F}_{\nu\xi}}{\nu\le \xi\le \zeta}}$ has almost all of the properties
of a small SMS and then show how to gently interleave limit levels of small cofinality
where necessary.

We start with the structure of the sets ${\mathcal F}_{\xi\xi+1}$ for $\xi<\zeta$.

Given $\xi<\zeta$ let $\alpha<\mu$ be such that for some $i<\zeta^{p_{\alpha}}$ 
one has $x_\xi=[(\alpha,i)]_\sim$ and $x_{\xi+1}=[(\alpha,i+1)]_\sim$.
Then, by the definition of $\Pbb$, ${\mathcal F}^{p_\alpha}_{ii+1}$ is either a
singleton (consisting of a order-preserving map) or an amalgamation pair, and 
for all $\beta\in (\alpha,\mu)$ we have $k_{\alpha\beta}(i+1)=k_{\alpha\beta}(i) + 1$ and
${\mathcal F}^{p_\beta}_{k_{\alpha\beta}(i)k_{\alpha\beta}(i)+1}= {\mathcal F}^{p_\alpha}_{ii+1}$. 
Hence we must then have that ${\mathcal F}_{\xi\xi+1}= {\mathcal F}^{p_\alpha}_{ii+1}$ and is a
singleton (consisting of a order-preserving map) or an amalgamation pair.  
(Note that, by the definition of small SMS, if ${\mathcal F}_{\xi\xi+1}={\mathcal F}^{p_\alpha}_{ii+1} =\set{f_i}$ 
is a singleton then $f_i$ is not cofinal.)

Now we show  the ${\mathcal F}_{\nu\xi}$ have the required factorization property.

Suppose $\nu\le\tau\le\xi<\zeta'$ and $f\in {\mathcal F}_{\nu\tau}$ and $g\in {\mathcal F}_{\tau\xi}$. 
As $\mu$ is a limit ordinal there is some $\alpha<\mu$ with some 
$i\le j \le l\le \zeta^{p_\alpha}$ such that $k_\alpha(i)=\nu$, $k_\alpha(j)=\tau$ and $k_\alpha(l)=\xi$
(\emph{i.e.}, $x_\nu=[(\alpha,i)]_\sim$, $x_\tau=[(\alpha,j)]_\sim$ and $x_\xi=[(\alpha,l)]_\sim$),
$f\in {\mathcal F}^{p_{\alpha}}_{ij}$ and $g\in {\mathcal F}^{p_{\alpha}}_{jl}$. Thus $g\cdot f\in
{\mathcal F}^{p_{\alpha}}_{il}$, and hence $g\cdot f\in {\mathcal F}_{\nu\xi}$. 

Similarly, suppose $\nu\le\tau\le\xi<\zeta$ and $h\in {\mathcal F}_{\nu\xi}$.  
There is some $\alpha<\mu$ and $i\le j \le l\le \zeta^{p_{\alpha}}$ 
such that $k_\alpha(i)=\nu$, $k_\alpha(j)=\tau$, $k_\alpha(l)=\xi$
and $h\in {\mathcal F}^{p_{\alpha}}_{il}$. 
But then, as $S^{p_{\alpha}}$ is a small SMS, 
there are $g\in {\mathcal  F}^{p_{\alpha}}_{jl}$, $\subseteq {\mathcal F}_{\tau\xi}$, and 
$f\in {\mathcal F}^{p_{\alpha}}_{ij}$, $\subseteq {\mathcal F}_{\nu\tau}$, such that 
$g\cdot f= h$.

Now suppose that $\zeta=\zeta'$ and $\nu\le\tau<\zeta$. If 
$f\in {\mathcal F}_{\nu\tau}$ and $g\in {\mathcal F}_{\tau\zeta}$ there is some $\alpha<\mu$,
$i\le j \le l\le \zeta^{p_{\alpha}}$ and $h\in {\mathcal F}^{p_\alpha}_{j\zeta^{p_{\alpha}}}$
such that $k_\alpha(i)=\nu$, $k_\alpha(j)=\tau$, $f\in {\mathcal F}^{p_\alpha}_{ij}$ and 
$g=F^{-1}\cdot F^{p_{\alpha}} \cdot h$. Then $g\cdot f = F^{-1}\cdot F^{p_{\alpha}} \cdot (h\cdot f)$,
where $h\cdot f \in {\mathcal F}^{p_\alpha}_{i\zeta^{p_{\alpha}}} \subseteq {\mathcal F}_{\nu\tau}$, and 
hence $g\cdot f\in {\mathcal F}_{\nu\zeta}$.

Similarly, if $\zeta=\zeta'$, $\nu\le\tau<\zeta$ and $f\in {\mathcal F}_{\nu\zeta}$
there is some $\alpha<\mu$, $i\le j \le l\le \zeta^{p_{\alpha}}$,
$h\in {\mathcal F}^{p_\alpha}_{ij}$ and $g\in {\mathcal F}^{p_\alpha}_{j \zeta^{p_{\alpha}}}$ 
such that $k_\alpha(i)=\nu$, $k_\alpha(j)=\tau$, and $f=(F^{-1}\cdot F^{p_{\alpha}}\cdot g)\cdot h$.

Next, we show directedness at limits. 

Suppose $\nu$, $\tau<\xi<\zeta'$, $\xi$ is a limit ordinal, $f\in {\mathcal F}_{\nu\xi}$ and $g\in {\mathcal F}_{\tau\xi}$.
Let $\alpha<\mu$ be such that there are $i\le j< l\le \zeta^{p_{\alpha}}$ 
with $k_\alpha(i)=\nu$, $k_\alpha(j)=\tau$, 
$f\in {\mathcal F}^{p_\alpha}_{il}$ and $g\in {\mathcal F}^{p_\alpha}_{jl}$. 

If $l$ is a successor then ${\mathcal F}^p_{l-1l}$ is a singleton, say
$\set{h_{l-1}}$ (as otherwise $k_{\alpha\beta}(l)=k_{\alpha\beta}(l-1)+1$ for all
$\beta\in (\alpha,\mu)$ and hence $k_\alpha(l)=\xi$ is a successor). 
Furthermore, as $S^p$ is a SMS there are $f'\in {\mathcal F}^{p_{\alpha}}_{il-1}$ and $g'\in {\mathcal F}^{p_{\alpha}}_{jl-1}$ such that 
$f=h_{l-1}\cdot f'$ and $g=h_{l-1}\cdot g'$. But then 
$k_\alpha(l-1) \in [\nu\cup\tau,\xi)$ and $f'\in {\mathcal F}_{\nu k_\alpha(l-1)}$, 
$g'\in {\mathcal F}_{\tau k_\alpha(l-1)}$ and $h_{l-1}\in {\mathcal F}_{k_\alpha(l-1)\xi}$. 

If $l$ is a limit ordinal, by the definition of small SMS, 
there is some $m\in [\max(\set{i,j}), l)$, and
$h\in {\mathcal F}^{p_{\alpha}}_{ml}$, $f'\in {\mathcal F}^{p_{\alpha}}_{im}$ and $g'\in {\mathcal F}^{p_{\alpha}}_{jm}$ such that
$f = h \cdot f'$ and $g= h\cdot g'$. Hence $f'\in {\mathcal F}_{\nu k_\alpha(m)}$, 
$g'\in {\mathcal F}_{\tau k_\alpha(l-1)}$ and $h\in {\mathcal F}_{k_\alpha(m)\xi}$. 

Now suppose that $\nu$, $\tau<\zeta=\zeta'$, $f\in {\mathcal F}_{\nu\zeta}$ and $g\in {\mathcal F}_{\tau\zeta}$.
Then there is some $\alpha<\mu$ and $i\le j\le \zeta^{p_{\alpha}}$ 
with $k_\alpha(i)=\nu$, $k_\alpha(j)=\tau$, and $k_\alpha(\zeta^{p_{\alpha}})=\xi<\zeta$, and some maps
$f'\in {\mathcal F}_{i\zeta^{p_{\alpha}}}$ and $g'\in {\mathcal F}_{j\zeta^{p_{\alpha}}}$ with
$f=F^{-1}\cdot F^{p_{\alpha}}\cdot f'$ and $g=F^{-1}\cdot F^{p_{\alpha}}\cdot g'$. As 
$F^{-1}\cdot F^{p_{\alpha}}\in {\mathcal F}_{\xi\zeta}$ we are done.

We have now shown that $\tup{\tupof{\theta_\xi}{\xi\le \zeta},
\tupof{{\mathcal F}_{\nu\xi}}{\nu\le \xi\le \zeta}}$ has all of the
properties of a small SMS, except possibly the property that if
$\xi\le\zeta$ is a limit ordinal then $\cf(\theta_\xi)<\kappa$.  In
order to guarantee this last property, in general, we must add
additional levels to ${\mathcal E}$.

In order to see why consider the example of some $\tupof{\alpha_l}{l<\omega}$
in which for all $l<\omega$ we have $k_{\alpha_l}(i) = i$ for all $i<l$ but 
$\omega= k_{\alpha_l}(l)$ and $\cf(\theta^{p_{\alpha_l}}_l)=\kappa$.  
We then need to insert a `new' $\omega$-th level in order to ensure 
that the $\theta$ of the $\omega$-th level of the lower bound we are constructing 
does not have cofinality $\kappa$.

This interleaving of additional levels is not hard to do, 
but writing down the details formally, as we do in the following four paragraphs, 
is a little involved.

So, set $\zeta^*=\zeta+1$ if $\cf(\theta_{x_\zeta})=\kappa$, and otherwise let $\zeta^*=\zeta$.
For $\alpha<\mu$ we define  $k^*_\alpha:\zeta^{p_{\alpha}}\longrightarrow \zeta^*$ as follows. 
For each $i\le\zeta^{p_{\alpha}}$ let $\xi_i$ be the greatest limit ordinal such that 
$\xi_i\le k_\alpha(i)$ if it exists
and undefined otherwise. If $\xi_i$ is defined and $\cf(\theta_{\xi_i})=\kappa$ then 
$k^*_\alpha(i)=k_\alpha(i)+1$. Otherwise we take $k^*_\alpha(i)=k_\alpha(i)$. 

For $\xi\le\zeta^*$ such that $\xi\in\bigcup\setof{\rge{k^*_\alpha}}{\alpha<\mu}$ we let
$\theta^*_\zeta = \theta^{p_{\alpha}}_j$ for any pair $(\alpha,j)$
such that $k^*_{\alpha}(j)=\xi$. If $\nu\le\xi\le\zeta^*$ and $\nu$,
$\xi\in\bigcup\setof{\rge{k^*_\alpha}}{\alpha<\mu}$ let $\alpha<\mu$
be such that $\nu$, $\xi\in\rge{k^*_\alpha}$, let 
$\nu'= k_\alpha((k^*_{\alpha})^{-1}(\nu)$ and 
$\xi'= k_\alpha((k^*_{\alpha})^{-1}(\xi)$ and set 
${\mathcal F}^*_{\nu\xi}= {\mathcal F}_{\nu'\xi'}$.

Now suppose otherwise, that $\xi\le\zeta^*$ but 
$\xi\ni\bigcup\setof{\rge{k^*_\alpha}}{\alpha<\mu}$, hence $\xi$ is a limit ordinal.
Let $\alpha^*$ be least such that $\xi\in\rge(k_{\alpha^*})$.  
For all $\alpha\in [\alpha^*,\mu)$ and all 
$f\in {\mathcal F}^{p_\alpha}_{ij}$, where $k_\alpha(j)=\xi$ let $\gamma_f=\ssup(\rge(f))$.
Lemma (\ref{individual-SMS-maps-are-not-cofinal}) shows that for all such $f$ we have $\gamma_f<\theta_\xi$. 
Hence, as $\cf(\theta_\xi)=\kappa$, we have 
$\ssup(\setof{\gamma_f}{\alpha\in [\alpha^*,\mu)\sss\&\sss
k_\alpha(j)=\xi \sss\&\sss f\in {\mathcal F}^{p_\alpha}_{ij}})<\theta_\xi$. 

So we can choose $\theta^*_\xi$ such that
$\ssup(\setof{\gamma_f}{\alpha\in [\alpha^*,\mu)\sss\&\sss
k_\alpha(j)=\xi \sss\&\sss f\in {\mathcal F}^{p_\alpha}_{ij}})<\theta^*_\xi<\theta_\xi$ and $\cf(\theta^*_\xi)<\kappa$.
We set ${\mathcal F}^*_{\xi\xi+1} = \set{\id}$
and define the other sets of maps in the 
obvious way using composition to ensure the last clause in the definition of small SMS holds: 
\emph{e.g.}, for $\nu\in\bigcup\setof{\rge{k^*_\alpha}}{\alpha<\mu}$
if $\nu<\xi$ we take
${\mathcal F}^*_{\nu\xi+1}=\setof{\id\cdot f}{f\in {\mathcal F}^*_{\nu\xi}}$
and if $\xi<\nu$ we take 
${\mathcal F}^*_{\xi\nu}=\setof{f\cdot \id}{f\in {\mathcal F}^*_{\xi+1\nu}}$.

Now set $S=\tup{\tupof{\theta^*_\xi}{\xi\le \zeta^*}, 
\tupof{{\mathcal F}^*_{\nu\xi}}{\nu\le \xi\le \zeta^*}}$.
It is clear that $S$ is an small SMS.

We also define $\mathcal{N}=\setof{N}{\exists \alpha<\mu \sss\sss  N\in \mathcal{N}^{p_\alpha}}$
and set $p=(S,F,\mathcal{N})$.

In order to finish the proof we must show that $p\in\Pbb$ and that
$p\le p_\alpha$ for each $\alpha<\mu$.

For the former all that is left to show is that $\mathcal{N}$ is as required. 

Suppose $M\in \mathcal{N}$. Let $\alpha<\mu$ be such that $M\in \mathcal{N}^{p_\alpha}$., and 
let $i^M\le \zeta^{p_\alpha}$ and $f\in {\mathcal F}^{p_\alpha}_{i\zeta^{p_\alpha}}$ witness this.
Then $M$ fits $F^{p_\alpha} \cdot f $, $=F\cdot (F^{-1}\cdot F^{p_\alpha} \cdot f )$ and the 
uniqueness of $(k^*_\alpha(i^M), F^{-1}\cdot F^{p_\alpha} \cdot f)$ follows from
Lemma (\ref{individual-SMS-maps-are-not-cofinal}). 

Suppose $i'\le j'\le k^*_\alpha(i^M)$ and $f\in {\mathcal F}_{i'j'}$. 
If $i'\in \bigcup\setof{\rge(k^*_\beta)}{\beta\in (\alpha\mu)}$ let $i=i'$ and otherwise
let $i=i'+1$. Define $j$ similarly. Then 
there is some $\beta\in (\alpha,\mu)$ and $\obar{i}$, $\obar{j}\le \zeta^{p_\beta}$
such that $k(\obar{i})=i$, $k(\obar{j})=j$ and $f\in {\mathcal F}^{p_\beta}_{\obar{i}\obar{j}}$. 
As $M\in {\mathcal N}^{p_\beta}$ we have that $f\in \obar{M}$.

Similarly, if $N\in {\mathcal N}$, $\gamma<\delta^N$, $g\in (\theta_{i^N})^\gamma$,
$i^N\le i^M$ and $g\in \obar{N}$ there is some $\beta\in (\alpha,\mu)$ such that
$N\in {\mathcal N}^{p_\beta}$ and $i^N\in\rge(k^*_\beta)$ and hence $g\in \obar{M}$.

Now let $\alpha<\mu$. We check that $p\le p_\alpha$.

We take the witnessing '$k$' to be $k^*_\alpha$. The three bulletted properties of $k^*_\alpha$ 
are immediate from the definition of $k^*_\alpha$ and that of $S$. We take
$f^{p_\alpha p}$ to be $F^{-1}\cdot F^{p_\alpha}$, making it immediate that $F^{p_\alpha} = 
F\cdot f^{p_\alpha p}$. It is also immediate from the definition of ${\mathcal N}$ that
${\mathcal N}^{p_\alpha}\subseteq {\mathcal N}$. 

Finally, if $N\in {\mathcal N}$ is witnessed by $(k^*_\alpha(i),f^{p_\alpha p}\cdot g)$ for
some $g\in {\mathcal F}^{p_\alpha}_{i\zeta^{p_\alpha}}$ then 
we must show that $N\in {\mathcal N}^p$.

Let $\beta \in [\alpha,\mu)$ be such that $N \in {\mathcal N}^{p_\beta}$. By the definition 
of $k^*_\alpha$ and the fact that $k^*_\alpha = k^*_\beta \cdot k_{\alpha\beta}$,
and since $F^{p_\alpha} = F^{p_\beta}\cdot f^{p_\alpha p_\beta}$, 
we have that this is witnessed by $(k_{\alpha\beta}(i),f^{p_\alpha p_\beta}\cdot g)$. 
As $p_\beta \le p_\alpha$ we have (by the final property in the list of requirements
defining $\le_{\Pbb}$) that $N\in {\mathcal N}^{p_\alpha}$ (and 
$(i,g)$ witnesses that $N$ fits $F^g\cdot g$). 
\end{proof}

\section{Showing $\Pbb$ adds an $(\kappa^+,1)$-simplified morass.}\label{proof-we-add-a-SM-section}
\vskip12pt

\begin{lemma}\label{density-lemma} Let $p\in \Pbb$, $\theta\in (\otp(\rge(F^p)\cup\set{\zeta})), \kappa^+)$
and $\zeta<\lambda$. 
Then there is some $q\in \Qbb$ such that $\zeta^q=\zeta^p+1$,  
$\theta=\theta^q_{\zeta^q}$ and $\zeta\in \rge(F^q)$.\end{lemma}

\begin{proof}[\ref{density-lemma}] Define $q$ as follows. Let $\zeta^q=\zeta^p+1$.

If $\zeta\ni \rge(F^p)$ let $\gamma=\otp(\rge(F^p)\cap \zeta))$.
Define $f_{\zeta^p}:\theta^p_{\zeta^p}\longrightarrow \theta$ by
$f_{\zeta^p}\on\gamma=\id$ and $f_{\zeta^p}(\gamma+\xi)=\gamma+1+\xi$ for
$\xi\in\otp(\theta^p_{\zeta^p}\setminus \gamma)$.
Define $F^q:\theta\longrightarrow \lambda$
by $F^q(\xi)=F^p(\xi)$ for $\xi<\gamma$, $F^q(\gamma)=\zeta$,
$F^q(\gamma+1+\xi) =F^p(\gamma+\xi)$ for
$\xi\in \otp(\theta^p_{\zeta^p}\setminus \gamma)$ 
and $F^q(\theta^p_{\zeta^p} + \xi)= \ssup(F^p``\theta^p_{\zeta^p})+\xi$ for
$\xi\in \otp(\theta\setminus(\gamma+1+\otp(\theta^p_{\zeta^p} \setminus \gamma)))$.

If, on the other hand, $\zeta\in \rge(F^p)$ let
$f_{\zeta^p}=\id$ and let $F^q:\theta\longrightarrow \lambda$ be
given by $F^q(\xi)=F^p(\xi)$ for $\xi<\theta^p_{\zeta^p}$ and 
$F^q(\theta^p_{\zeta^p} + \xi)= \ssup(F^p``\theta^p_{\zeta^p})+\xi$ for 
$\xi\in \otp(\theta\setminus\theta^p_{\zeta^p})$. 

Finally, let 
\[S^q=\tup{ \tupof{ \theta^p_{i} }{ i\le \zeta^p }\concat 
\theta,
\tupof{ {\mathcal F}^p_{ij} }{ i\le j\le \zeta^p }\concat
\tupof{ \setof{f_{\zeta^p} \cdot f }{ f\in {\mathcal F}^p_{i\zeta^p} } }{i\le \zeta^p} },\] 
${\mathcal F}^*=\setof{ (F^q\cdot f, \theta^{q}_i) }{ 
i\le \zeta^p+1 \sss\&\sss f\in {\mathcal F}^q_{i\zeta^p+1} }$ and 
${\mathcal N}^q={\mathcal N}^p$.\end{proof}

\vskip12pt

\begin{proposition}\label{P-adds-a-simplified-morass} Forcing with $\Pbb$ adds an 
$(\kappa^+,1)$-simplified morass.\end{proposition}

\begin{proof}[\ref{P-adds-a-simplified-morass}] Let $G$ be $\Pbb$-generic. 
Define two relations 
$\sim$ and $<$ on the set $\setof{(p,i)}{p\in G\sss\&\sss i\le \zeta^p}$,
by $(p,i)\sim (q,j)$ if there is some $r\in G$ with $r\le p$, $q$ 
witnessed by $k:\zeta^p\longrightarrow \zeta^r$ and $l:\zeta^q\longrightarrow \zeta^r$
and $k(i)=l(j)$ and, similarly, $(p,i) < (q,j)$ 
if there is some $r\in G$ with $r\le p$, $q$ 
witnessed by $k:\zeta^p\longrightarrow \zeta^r$ and $l:\zeta^q\longrightarrow \zeta^r$
and $k(i)<l(j)$. Clearly $\sim$ is an
equivalence relation and $<$ is a linear ordering which respects
$\sim$, inducing a linear ordering on the $\sim$-equivalence classes
${\mathcal E}=\setof{[(p,i)]_\sim}{p\in G\sss\&\sss i\le \zeta^p}$.  

For $x=[(p,i)]_\sim\in {\mathcal E}$ let $\theta_x = \theta^p_i$. Note, by
the definition of $\sim$ and of $\le_{\Pbb}$, that $\theta_x$ is
independent of the choice of the pair $(p,i)$. Note, also, that $x = [
p,i)]_{\sim} < y = [(q,j)]_{\sim}$ if and only if
$\theta_x<\theta_y<\kappa^+$. Consequently, $<$ well-orders ${\mathcal E}$
and, by Lemma (\ref{density-lemma}), it does so in order-type $\kappa^+$, since 
$\setof{\theta_x}{x\in {\mathcal E}}$ is unbounded in $\kappa^+$.
Let $\setof{x_{\alpha}}{\alpha<\kappa^+}$ be the $<$-increasing
ordering of ${\mathcal E}$. 

For $x$, $y\in {\mathcal E}$ define
${\mathcal F}_{xy}=\bigcup\setof{f\in {\mathcal F}^p_{ij}}{p\in G
\sss\&\sss x=[(p,i)]_\sim \sss\&\sss y=[(p,j)]_\sim }$.
Let ${\mathcal F}_x= \setof{F^p\cdot f}{p\in G\sss\&\sss
x=[(p,i)]_\sim \sss\&\sss f\in {\mathcal F}^p_{i\zeta^p} }$.

For $\alpha<\kappa^+$ let $\theta_\alpha=\theta_{x_\alpha}$. 
Let $\theta_{\kappa^+}=\lambda$. For $\alpha\le\beta<\kappa^+$ let
${\mathcal F}_{\alpha\beta}={\mathcal F}_{x_{\alpha}x_{\beta}}$. For
$\alpha<\kappa^+$ let ${\mathcal F}_{\alpha\kappa^+}= 
{\mathcal F}_{x_\alpha}$. 

It remains to check that 
$M=\tup{\tupof{\theta_\alpha}{\alpha\le \kappa^+}, 
\tupof{{\mathcal F}_{\alpha\beta}}{\alpha\le \beta\le \kappa^+}}$ 
is an $(\kappa^+,1)$-simplified morass.

Given $\alpha<\kappa^+$ let $p\in G$ be such that for some $i<\zeta^p$ 
one has $x_\alpha=[(p,i)]_\sim$ and $x_{\alpha+1}=[(p,i+1)]_\sim$.
Then, by the definition of $\Pbb$, ${\mathcal F}^p_{ii+1}$ is either a
singleton (consisting of a order-preserving map) or an amalgamation pair, 
and by the definition of
$\le_{\Pbb}$, one must then have that ${\mathcal F}_{\alpha\alpha+1}$ is a
singleton (consisting of a order-preserving map) or an amalgamation
pair.  

Suppose $\alpha\le\beta\le\gamma\le\kappa^+$ and $f\in {\mathcal
F}_{\alpha\beta}$ and $g\in {\mathcal F}_{\beta\gamma}$. If $\gamma< \kappa^+$ then by the
directedness of $G$ there is some $p\in G$ and 
$i\le j \le k\le \zeta^p$ such that $x_\alpha=[(p,i)]_\sim$, 
$x_\beta=[(p,j)]_\sim$, $x_\gamma=[(p,k)]_\sim$,
$f\in {\mathcal F}^p_{ij}$ and $g\in {\mathcal F}^p_{jk}$. Thus $g\cdot f\in
{\mathcal F}^p_{ik}$, and hence $g\cdot f\in {\mathcal F}_{\alpha\gamma}$. 
Similarly, if $\gamma=\kappa^+$ there is some $p\in G$ and 
$i\le j \le \zeta^p$ such that $x_\alpha=[(p,i)]_\sim$, $x_\beta=[(p,j)]_\sim$, 
$f\in {\mathcal F}^p_{ij}$ and $g = F^p\cdot g'$ where $g'\in {\mathcal F}^p_{j\zeta^p}$.
Thus $g\cdot f = F^p\cdot g'\cdot f\in {\mathcal F}_{\alpha\kappa^+}$. 

Similarly, suppose $\alpha\le\beta\le\gamma\le\kappa^+$ and $h\in {\mathcal
F}_{\alpha\gamma}$. If $\gamma<\kappa^+$ there is some $p\in G$ and $i\le j \le k\le \zeta^p$ 
such that $x_\alpha=[(p,i)]_\sim$, 
$x_\beta=[(p,j)]_\sim$, $x_\gamma=[(p,k)]_\sim$ and $h\in {\mathcal
F}^p_{ik}$. But then, as $S^p$ is a small SMS, 
there are $g\in {\mathcal  F}^p_{jk}$, $\subseteq {\mathcal F}_{\beta\gamma}$, and 
$f\in {\mathcal F}^p_{ij}$, $\subseteq {\mathcal F}_{\alpha\beta}$, such that 
$g\cdot f= h$. Likewise, if $\gamma=\kappa^+$ there is some 
$p\in G$ and $i\le j \le \zeta^p$ such that $x_\alpha=[(p,i)]_\sim$, 
$x_\beta=[(p,j)]_\sim$, and $h = F^p \cdot h'$ with $h'\in {\mathcal F}^p_{i\zeta^p}$.
Then, once more, as $S^p$ is a small SMS, there are $g\in {\mathcal  F}^p_{j\zeta^p}$ and 
$f\in {\mathcal F}^p_{ij}$, $\subseteq {\mathcal F}_{\alpha\beta}$, such that 
$g\cdot f= h'$, and so $F^p\cdot g\in {\mathcal F}_{\beta\kappa^+}$.

Directedness at limits: suppose $\alpha\le\beta<\varepsilon\le\kappa^+$, 
$f\in {\mathcal F}_{\alpha\varepsilon}$ and $g\in {\mathcal F}_{\beta\varepsilon}$, and
$\varepsilon$ is a limit ordinal.

First of all suppose $\varepsilon<\kappa^+$. 
By the directedness of $G$ let $p\in G$ be such that there are $i\le
j< k\le \zeta^p$ with such that $x_\alpha=[(p,i)]_\sim$, 
$x_\beta=[(p,j)]_\sim$, $x_\varepsilon=[(p,k)]_\sim$,
$f\in {\mathcal F}^p_{ik}$ and $g\in {\mathcal F}^p_{jk}$. 

If $\cf(\varepsilon)=\kappa$ then $k$ is a successor and ${\mathcal F}^p_{k-1k}$ is a singleton, say
$\set{h_{k-1}}$. Furthermore, as $S^p$ is a SMS there are $f'\in {\mathcal
F}^p_{ik-1}$ and $g'\in {\mathcal F}^p_{jk-1}$ such that 
$f=h_{k-1}\cdot f'$ and $g=h_{k-1}\cdot g'$. But then 
there is some $\gamma \in [\alpha\cup\beta,\varepsilon)$ such that 
$x_\gamma=[(p,k-1)]_\sim$ and $f'\in {\mathcal F}_{\alpha\gamma}$, 
$g'\in {\mathcal F}_{\beta\gamma}$ and $h_{k-1}\in {\mathcal
F}_{\gamma\varepsilon}$. 

If $\cf(\varepsilon)<\kappa$ we may suppose by the closure of $\Pbb$ and genericity that $k$ is a limit ordinal. 
Whence by the definition of SMS 
that there are $l\in [\max(\set{i,j}), k)$, with $x_\gamma = [(p,l)]_{~}$,
$h\in {\mathcal F}_{lk}$, $f'\in {\mathcal F}_{il}$ and $g'\in {\mathcal F}_{jl}$ such that
$f = h \cdot f'$ and $g= h\cdot g'$. Hence $f'\in {\mathcal F}_{\alpha\gamma}$, 
$g'\in {\mathcal F}_{\beta\gamma}$ and $h\in {\mathcal F}_{\gamma\varepsilon}$. 

If $\varepsilon=\kappa^+$ then, again by the directedness of $G$, there is some 
$p\in G$ such that there are $i\le j< \zeta^p$ such that $x_\alpha=[(p,i)]_\sim$, 
$x_\beta=[(p,j)]_\sim$, and there are 
$f'\in {\mathcal F}^p_{i\zeta^p}$ and $g'\in {\mathcal F}^p_{j\zeta^p}$ with
$f= F^p\cdot f'$ and $g= F^p\cdot g'$. 

Thus we have verified directedness at $\varepsilon$.

We have that $\bigcup\setof{f``\theta_\alpha}{\alpha<\kappa^+
\sss\&\sss f\in {\mathcal F}_{\alpha\kappa^+}} = \lambda$ by Lemma (\ref{density-lemma}).

Finally, we shall show that $V[G]\thinks \hat{\lambda} = \name{\kappa}_2$.

We recall a useful lemma of Velleman's and a couple of related definition. 

\begin{lemma}\label{Velleman-lemma} (\cite{Velleman-SM}, Lemma (3.2).)
If $\alpha\le\beta\le\kappa^+$, $f_0$, $f_1\in {\mathcal F}_{\alpha\beta}$, $\tau_0$, $\tau_1<\theta_\alpha$ 
and $f_0(\tau_0)=f_1(\tau_1)$ then $\tau_0=\tau_1$ and $f_0\on\tau_0+1=f_1\on\tau_1+1$.
\end{lemma}

\begin{proof}[\ref{Velleman-lemma}] By induction on $\beta$ for each fixed $\alpha$, using the 
simplicity of the structure of the sets ${\mathcal F}_{\beta'\beta'+1}$ at successor steps and
the directedness at limits at limit steps.
\end{proof}

\begin{definition}\label{defn-the-psi-maps} (\cite{M96}) 
If $\alpha\le\beta\le\kappa^+$, $\tau'<\theta_\alpha$ and there is some 
$f\in {\mathcal F}_{\alpha\beta}$ and $\tau<\theta_\beta$ such that $f(\tau')=\tau$,
let $\psi_{(\alpha,\tau'),(\beta,\tau)} = f\on \tau' + 1$. 
\end{definition}

By Lemma (\ref{Velleman-lemma}) each such $\psi_{(\alpha,\tau'),(\beta,\tau)}$ is well defined.

\begin{definition}\label{defn-xi-alpha} (\cite{M96}) 
If $\alpha<\kappa^+$, $\tau<\lambda$ and there is some 
$f\in {\mathcal F}_{\alpha\kappa^+}$ with $\tau\in\rge(f)$ write $\tau_\alpha$ for the unique
$\tau'$ such that $f(\tau')=\tau$.
\end{definition}

Let $\xi<\lambda$. There is some $\alpha_\xi<\kappa^+$ such that for all 
$\alpha\in [\alpha_\xi,\kappa^+)$ there is some $\xi_\alpha<\theta_\alpha$ with
$\psi_{(\alpha,\xi_\alpha),(\kappa^+,\xi)}$ well defined. For such $\alpha$ set
$A_\alpha =\rge(\psi_{(\alpha,\xi_\alpha),(\kappa^+,\xi)}\on\xi_\alpha)$.

By directedness at the limit ordinal $\kappa^+$, we have that for every 
$\tau< \xi$ there is some $\alpha\in[\alpha^*,\kappa^+)$ and $f\in {\mathcal F}_{\alpha\kappa^+}$ 
such that $\tau$, $\xi\in\rge(f)$, and hence $\tau\in A_\alpha$. So
$\xi = \bigcup\setof{A_\alpha}{\alpha\in [\alpha_\xi,\kappa^+)}$ and thus is the union of the 
$\kappa^+$ many sets of size at most $\kappa$; hence we must have that $V[G]\thinks \hat{\lambda} = \name{\kappa}_2$.
$\hphantom{w}$\end{proof}

\vskip18pt

In \cite{SVV}, Shelah, V\"a\"an\"anen and Veli\v ckovi\'c introduced the antichain property for 
simplified morasses, a property related to prior work of Miyamoto (\cite{Miyamoto}).

\begin{definition}\label{defn-omega2-antichain-property} (\cite{SVV}) Let 
$M=\tup{\tupof{\theta_\alpha}{\alpha\le \kappa^+}, 
\tupof{{\mathcal F}_{\alpha\beta}}{\alpha\le \beta\le \kappa^+}}$ 
be a $(\kappa^+,1)$-simplified morass. As in Definition (\ref{defn-xi-alpha}),
if $\alpha<\kappa^+$, $\tau<\kappa^{++}$ and there is some 
$f\in {\mathcal F}_{\alpha\kappa^+}$ with $\tau\in\rge(f)$ write $\tau_\alpha$ for the unique
$\tau'$ such that $f(\tau')=\tau$. We say $M$ has the $\kappa^{++}$-\emph{antichain property} if
\begin{align*}\forall X\in [\kappa^{++}]^{\kappa^{++}} \sss \sss\exists\tau, & \sss \xi\in X \sss\sss
\forall \alpha<\kappa^+ \\
& (\tau_\alpha \hbox{ and }\xi_\alpha\hbox{ are both defined} \Longrightarrow \tau_\alpha \le \xi_\alpha).
\end{align*}
\end{definition}

We note that the analogue for $(\omega,1)$-simplified morasses always fails if Martin's Axiom holds:
whilst ZFC shows there are always $(\omega,1)$-simplified morasses,
MA$_{\omega_1}$ implies that no $(\omega,1)$-simplified morass has the $\omega_1$-antichain property.
As far as we know it is open whether there is a similar consistency result for 
$(\kappa^+,1)$-simplified morasses.

\begin{proposition}\label{the-generic-morass-has-the-omega2-antichain-property} 
The $(\kappa^+,1)$-simplified morass added by forcing with $\Pbb$ has the 
$\kappa^{++}$-antichain condition.
\end{proposition}

\begin{proof}[\ref{the-generic-morass-has-the-omega2-antichain-property}]
Suppose $p\in \Pbb$ and $p\forces \name{X}\in [\kappa^{++}]^{\kappa^{++}}$. We can find
some $E\in [\kappa^{++}]^{\kappa^{++}}$ and for each $\xi\in E$ some $p_\xi\in \Pbb$  
such that $p_\xi \forces  s\in \name{X}$. By Lemma(\ref{density-lemma}) we may 
assume that $\xi\in \rge(F^{p_\xi})$ for each $\xi\in E$. By the $\Delta$-system lemma we
may thin $E$, if necessary, so that for distinct $\xi$, $\xi'\in S$ we have 
$\dom(F^{p_\xi})=\dom(F^{p_{\xi'}})$, there is some $\obar{\xi}$ such that 
$F^{p_{\xi}}(\obar{\xi})=\xi$,  $F^{p_{\xi'}}(\obar{\xi})=\xi'$,
$\xi'\ni \rge(F^{p_\xi})$ and $\xi\ni \rge(F^{p_{\xi'}})$. 
Now apply the proof of Proposition (\ref{P-has-lambda-cc}) to the
collection of $p_\xi$ for $\xi$ in the thinned $E$. This gives us some $\xi$, $\xi'\in E$ with $\xi<\xi'$
and some $r\le p_\xi$, $p_{\xi'}$ such that 
$r\forces \forall \alpha<\kappa^+\sss (\xi_\alpha \hbox{ and }\xi'_\alpha\hbox{ are both defined} 
\Longrightarrow \xi_\alpha \le \xi'_{\alpha})$.
\end{proof}

\section{Summary of results proven}\label{results-proven}

\begin{theorem}\label{main-theorem} Suppose $\lambda$ is regular and $2^\kappa< \lambda$. 
There is a $\kappa^+$-proper, $\lambda$-chain condition forcing (hence preserving 
$\kappa^+$, $\lambda$ and all larger cardinals and cofinalities)
which collapses any cardinals $\mu$ with $\kappa^+<\mu<\lambda$, 
and such that in the forcing extension there is an $(\kappa^+,1)$-simplified morass
with the $\kappa^{++}$-antichain condition.
\end{theorem}

\begin{proof} Immediate from Propositions (\ref{P-is-proper}),
(\ref{P-has-lambda-cc}),
(\ref{P-adds-a-simplified-morass}) and (\ref{the-generic-morass-has-the-omega2-antichain-property}).
\end{proof}

\begin{corollary}\label{lambda-equal-omega2-case-of-main-theorem} If 
$2^{\kappa}= \kappa^+$ then there is 
a $\kappa^+$-proper, $\kappa^{++}$-cc forcing which preserves 
all cardinals and cofinalities and adds an
$(\kappa^+,1)$-simplified morass
with the $\kappa^{++}$-antichain condition.
\end{corollary}

\begin{proof}[\ref{lambda-equal-omega2-case-of-main-theorem}] By the premise that 
$2^{\kappa}= \kappa^+$ we may apply Theorem (\ref{main-theorem}) with 
$\lambda=\kappa^{++}$.
\end{proof}

\begin{corollary}\label{forcing-with-P-adds-a-Kurepa-tree-and-a-square-sequence} If $2^\kappa<\lambda$
the forcing $\Pbb$ of Theorem (\ref{main-theorem}) adds an 
$\kappa^+$-Kurepa tree. If $\kappa=\omega$ then $\Pbb$ adds $\squaresub{\omega_1}$ sequence.
\end{corollary}

\begin{proof}[\ref{forcing-with-P-adds-a-Kurepa-tree-and-a-square-sequence}]
Immediate from Theorem (\ref{main-theorem}) and the work of Velleman in \cite{Velleman-SM},
and \cite{Velleman-SM} and \cite{Velleman-MDF}, respectively, where it is shown that 
if there is a $(\kappa^+,1)$-simplified morass then there is a Kurepa tree, and
if there is a $(\omega_1,1)$-simplified morass then $\squaresub{\omega_1}$ holds. 
\end{proof}

\begin{remark} That, when $\kappa=\omega$, $\Pbb$ adds a Kurepa tree should be 
contrasted with the old theorem of Jensen that if $\squaresub{\omega_1}$ holds there is a ccc forcing
-- with finite working parts -- to add an $\omega_1$-Kurepa tree. See
\cite{Velickovic} for a proof of this result in the style of the
introductory remarks to this paper.

However there can be no direct extension of that theorem to a one that asserts that if 
$\squaresub{\omega_1}$ there is ccc forcing to add an $(\omega_1,1)$-simplified 
morass. This is because $\squaresub{\omega_1}$ is consistent with the weak Chang conjecture holding,
no ccc forcing can destroy the truth of the weak Chang conjecture and the existence of 
an $(\omega_!,1)$-simplified morass implies the failure of the weak Chang conjecture.
\end{remark}
\vskip12pt

In future work we shall address higher gap analogues of these results.


\vskip18pt


\begin{thebibliography}{1}


\bibitem{Baumgartner-Shelah}
Baumgartner, J., and Shelah, S., {\it Remarks on superatomic Boolean
algebras\/}, Annals of Pure and Applied Logic {\bf 33} (1987), pp.~109-129.

\bibitem{Donder} Donder, H.-D., {\it Another look at gap-one morasses\/}, 
Symposium of Pure Mathematics, {\it Recursion Theory\/},
eds.~Nerode, A., and Shore, R., 
Proceedings of Symposia in Pure Mathematics, {\bf 42}, American
Mathematical Society, (1985), pp.~223-236.

\bibitem{Friedman} Friedman, S., {\it Forcing with finite conditions\/}, in {\it Set
Theory: Centre de Recerca Matemàtica, Barcelona, 2003-2004\/},
eds.~Bagaria, J., and Todorcevic, S., Trends in Mathematics, Birkhäuser
Verlag, pp.~285-295, 2006. 

\bibitem{GM} Gitik, M., and Magidor, M,  {\it SPFA by finite conditions\/}, preprint

\bibitem{Gitik} Gitik, M., {\it  A certain generalization of SPFA to higher cardinals\/}, preprint.

\bibitem{Hyttinen-Rautila} Hyttinen, T.,  and Rautila, M., {\it The canary tree revisited\/}, The Journal of
Symbolic Logic, {\bf 66} (2001), pp.~1677-1694.

\bibitem{Koszmider} Koszmider, K., {\it On the existence of strong chains in
P($\omega$)/Fin\/}, Journal of Symbolic Logic. {\bf 63}(3), (1998), 
pp.~1055-1062. 

\bibitem{Mitchell1} Mitchell, W., {\it A weak variation of Shelah's I[$\omega_2$]\/} 
Journal of Symbolic Logic, {\bf 69}(1), (2004) , pp.~94-100.

\bibitem{Mitchell2} Mitchell, W., {\it Adding closed unbounded subsets of $\omega_2$ with
Finite Forcing\/}, Notre Dame Journal of Formal Logic, {\bf 46}(3), 
(2005), pp.~357-371.

\bibitem{Mitchell3} Mitchell, W., {\it $I[\omega\sb 2]$ can be the nonstationary
ideal on ${\rm Cof}(\omega\sb 1)$\/}.
Transactions of the American Mathematical Society {\bf 361}(2) (2009), pp.~561-601. 

\bibitem{Mitchell4} Mitchell, W., {\it Notes on a proof of Koszmider\/}, unpublished
note, January 2004. {\tt http://www.math.ufl.edu/~wjm/papers/koszmider.pdf\/}

\bibitem{Miyamoto} Miyamoto, T., {\it Simplified morasses which capture the $\Delta$ systems},
technical report, Nanzan University, 1991, 9pp.

\bibitem{M96} Morgan, C., {\it Morasses, square and forcing axioms\/}, Annals
of Pure and Applied Logic, {\bf 80}, (1996), pp.~139-163.

\bibitem{M06} Morgan,  C., {\it Local connectedness and distance functions\/}, 
in {\it Set Theory: Centre de Recerca Matemàtica, Barcelona,
2003-2004,\/}, eds.~Bagaria, J., and Todorcevic, S., Trends in
Mathematics, Birkhäuser Verlag, pp.~345-400, 2006.

\bibitem{M*clubs} Morgan, C., {\it Adding clubs with forcing with small working
parts\/}, Centre de Recerca Matemàtica preprint, 2003.

\bibitem{Neeman} Neeman, I., {\it  Forcing with sequences of models of two types\/}, preprint.

\bibitem{Roitman} Roitman, J., {\it Height and width of superatomic Boolean
algebras\/}, Proceedings of the American Mathematical Society, {\bf 94},
(1985), pp.~9-14.

\bibitem{Roslanowski-Shelah} Roslanowski, A., and Shelah, S.,
{\it The last forcing standing with diamonds}, preprint, arXiv:1406.4217v1.

\bibitem{SVV} Shelah, S., V\"a\"an\"anen, J., and Veli\v ckovi\'c, B.,
{\it Positional strategies in long Ehrenfeucht-Fra\"iss\'e games}, preprint.

\bibitem{Todorcevic} Todorcevic, S., {\it Partition problems in topology}, 
Contemporay Mathematics series, {\bf 84}, American Mathematical Society, 1989.

\bibitem{Velleman-MDF} Velleman, D., {\it Morasses, diamond, and forcing\/}, 
Annals of Mathematical Logic {\bf 23} (1982), pp.~199~28I

\bibitem{Velleman-SM} Velleman, D., {\it Simplified morasses\/}, 
Journal of Symbolic Logic, {\bf 49}(1), (1984), pp.~94-100.

\bibitem{Velleman-SMLL} Velleman, D., {\it Simplified morasses with linear limits\/}, 
Journal of Symbolic Logic, {\bf 49}(4), (1984), pp.~1001-1021

\bibitem{Velickovic} Veli\v ckovi\'c, B., {\it Stationary sets and forcing axioms\/},
Advances in Mathematics {\bf 94}, (1992), pp.~256-284.

\end{thebibliography}
\end{document}